\documentclass[12pt]{article}
 
\newcommand{\ignore}[1]{}

\usepackage{amsmath, amsthm, amsfonts, amssymb} 

\usepackage[dvips,final]{graphicx}

\usepackage{graphicx}

\usepackage{float}
\newfloat{figure}{H}{lof}
\floatname{figure}{\figurename}

\usepackage[colorlinks=true, urlcolor=blue,linkcolor=blue, citecolor=blue]{hyperref}

\numberwithin{figure}{section}

\DeclareMathAlphabet{\eufrak}{U}{}{}{}  
\SetMathAlphabet\eufrak{normal}{U}{euf}{m}{n}
\SetMathAlphabet\eufrak{bold}{U}{euf}{b}{n}

\numberwithin{equation}{section}

\newenvironment{Proof}{\removelastskip\par\medskip
\noindent{\em Proof.}
\rm}{\hfill$\square$\\\par\medbreak}

\def\trace{{\mathord{{\rm trace ~}}}}

\def\real{{\mathord{{\rm I\kern-2.8pt R}}}}        
\def\inte{{\mathord{{\rm I\kern-2.8pt N}}}}
\def\PP{{\mathord{{\rm I\kern-2.8pt P}}}}

\def\real{{\mathord{\mathbb R}}}

\def\inte{{\mathord{\mathbb N}}}

\def\Dom{{\mathrm{{\rm Dom}}}}

\def\grad{{\mathrm{ \ \! {\rm grad}}}}
\def\crb{{\mathrm{{\rm R}}}}

\newcommand{\disp}{\displaystyle}

\newcommand{\vide}{{}}

\def\P{\mathbb{P}}
\def\E{\mathop{\hbox{\rm I\kern-0.20em E}}\nolimits}

\newtheorem{prop}{Proposition}[section]

\newtheorem{lemma}[prop]{Lemma}
\newtheorem{definition}[prop]{Definition}

\newtheorem{remark}[prop]{Remark}
\def\Dom{\rm Dom \ \! }
\def\trace{{\mathrm{{\rm trace}}}}

\def\var{{\mathrm{{\rm Var}}}}

\def\argmax{{\mathrm{{\rm argmax}}}}

\title{ 
\Huge 
 Stochastic analysis on Gaussian space applied to drift estimation 
} 
 
\author 
{ 
Nicolas Privault\footnote{nicolas.privault@math.univ-poitiers.fr} 
\\ 
Laboratoire de Math\'ematiques 
\\ 
Universit\'e de Poitiers 
\\ 
T\'el\'eport 2 - BP 30179 
\\ 
86962 Chasseneuil Cedex 
\\ 
France 
\and 
Anthony R\'eveillac\footnote{anthony.reveillac@univ-lr.fr}  
\\ 
Laboratoire de Math\'ematiques 
\\ 
Universit\'e de La Rochelle 
\\ 
Avenue Michel Cr\'epeau 
\\ 
17042 La Rochelle Cedex
\\ 
France 
} 

\date{august 29, 2007}

\allowdisplaybreaks

\begin{document}
\hyphenation{func-tio-nals} 
\maketitle
 
\begin{abstract} 
 In this paper we consider the nonparametric 
 functional estimation of the drift 
 of Gaussian processes using Paley-Wiener 
 and Karhunen-Lo\`eve expansions. 
 We construct efficient estimators for 
 the drift of such processes, 
 and prove their minimaxity using Bayes estimators. 
 We also construct superefficient estimators of Stein type 
 for such drifts using the Malliavin integration by parts formula 
 and stochastic analysis on Gaussian space, in which superharmonic 
 functionals of the process paths play a particular role. 
 Our results  are illustrated by numerical simulations 
 and extend the construction of James-Stein type estimators 
 for Gaussian processes by Berger and Wolpert \cite{berger1}. 
\end{abstract} 
 
\normalsize

\vspace{0.5cm}


\normalsize

\baselineskip0.7cm

\section{Introduction} 

\noindent 
 The maximum likelihood estimator $\hat{\mu}$ of 
 the mean $\mu\in \real^d$ of a Gaussian random vector $X$ in $\real^d$ 
 with covariance $\sigma^2 {\rm I}_{\real^d}$ under a probability $\P_\mu$ 
 is well-known to be equal to $X$ itself, 
 and can be computed by maximizing the likelihood ratio 
$$ 
\frac{1}{(2\pi\sigma^2)^{d/2}} e^{-\frac{\Vert X-m \Vert^2_d}{2\sigma^2}}
$$ 
 with respect to $m$, 
 where $\Vert \cdot \Vert_d$ denotes the Euclidean norm on $\real^d$. 
 It is efficient in the sense that it attains the Cramer-Rao bound
$$ 
 \sigma^2 d 
 = 
 \E_\mu [ \Vert X - \mu \Vert_d^2 ] 
 = \inf_Z 
 \E_\mu [ \Vert Z - \mu \Vert_d^2 ] 
, \qquad 
 \mu \in \real^d, 
$$ 
 over all unbiased estimators $Z$ satisfying 
 $\E_\mu [Z] = \mu$, for all $\mu \in \real^d$. 
\\ 
 
\noindent 
 In \cite{jamesstein}, James and Stein have constructed 
 superefficient estimators for the mean of $X \in \real^d$, 
 of the form 
$$ 
 \left( 
 1- \frac{d-2}{\Vert X \Vert_d^2} 
 \right) 
 X 
$$ 
 whose risk is lower than the Cramer Rao bound $\sigma^2 d$ 
 in dimension $d\geq 3$. 
\\ 

\noindent 
 Drift estimation for Gaussian processes is of interest in several fields of 
 application. 
 For example in the decomposition 
$$ 
 X_t=X^u_t+u_t, \qquad t \in [0,T], 
$$ 
 the process $(X_t)_{t\in [0,T]}$ is interpreted as 
 an observed output signal, the drift $(u_t)_{t \in [0,T]}$ 
 is viewed as an input signal to be estimated and perturbed by a 
 centered Gaussian noise $(X^u_t)_{t \in [0,T]}$, 
 cf. e.g. \cite{IbragimovRozanov}, Ch.~VII. 
 Such results find applications in e.g. telecommunication 
 (additive Gaussian channels) and finance (identification of market trends). 
\\ 

\noindent 
 Berger and Wolpert \cite{berger1}, \cite{berger2}, have constructed estimators 
 of James-Stein type for the drift of a Gaussian process $(X_t)_{t\in [0,T]}$ 
 by applying the James-Stein procedure to the independent Gaussian random variables 
 appearing in the Karhunen-Lo\`eve expansion of the process. 
 In this context, $\hat{u} : = (X_t)_{t\in \real_+}$ is seen as a minimax 
 estimator of its own drift $(u_t)_{t\in \real_+}$. 
\\ 

\noindent 
 Stein~\cite{stein} has shown that the James-Stein estimators on $\real^d$ 
 could be extended to a wider family of estimators, using integration by parts for 
 Gaussian measures. 
 Let us briefly recall Stein's argument, which relies on integration by parts
 with respect to the Gaussian density 
 and on the properties of superharmonic functionals for the Laplacian on $\real^d$. 
 Given an estimator of $\mu \in \real^d$ of the form $X+g(X)$, where 
 $g: \real^d \to \real^d$ is sufficiently smooth,
 and applying the integration by parts formula 
\begin{equation} 
\label{ibp2} 
 \E_\mu [(X_i - \mu_i ) g_i (X)] = 
 \sigma^2 
 \E_\mu [\partial_i g_i (X)], 
\end{equation} 
 $g = \sigma^2 \grad \log f = \sigma^2 ( \partial_1 \log f , \ldots , \partial_d \log f )$, 
 one obtains 
$$ 
\E_\mu [\Vert X + 
 \sigma^2 \grad \log f (X) 
 - 
 \mu \Vert_d^2] = \sigma^2 d 
 + 
 4 \sigma^4 
 \sum_{i=1}^d 
 \E_\mu \left[ \frac{ 
 \partial^2_i \sqrt{f} (X)}{\sqrt{f} (X)} \right] 
, 
$$ 
 i.e. $X+ \sigma^2 \grad \log f (X)$ is a superefficient estimator if 
$$ 
 \sum_{i=1}^d \partial^2_i \sqrt{f} (x) < 0, \qquad dx-a.e. 
, 
$$ 
 which is possible if $d\geq 3$. 
 In this case, $X+ \sigma^2 \grad \log f (X)$ improves in the mean square sense
 over the efficient estimator $\hat{u}$ which 
 attains the Cramer-Rao bound $\sigma^2 d$ on unbiased estimators of $\mu$. 
\\ 
 
\noindent 
 In this paper we present an extension of Stein's argument to an 
 infinite-dimensional setting using the Malliavin integration by 
 parts formula, with application to the construction of Stein type 
 estimators for the drift of a Gaussian process $(X_t)_{t\in [0,T]}$. 
 Our approach applies to Gaussian processes such as Volterra processes 
 and fractional Brownian motions. 
 It also extends the results of Berger and Wolpert \cite{berger1} 
 in the same way that the construction of Stein~\cite{stein} extends that 
 of James and Stein~\cite{jamesstein}, 
 and this allows us to recover the estimators of James-Stein type 
 introduced by Berger and Wolpert \cite{berger1} as particular cases. 
 Here we replace the Stein equation 
 \eqref{ibp2} with the integration by parts formula of the Malliavin 
 calculus on Gaussian space. Our estimators are given by processes of the 
 form 
$$ 
 X_t + D_t \log F, \qquad 
 t\in [0,T] 
, 
$$ 
 where $F$ is a positive superharmonic random variable on 
 Gaussian space and $D_t$ is the Malliavin derivative indexed by $t\in [0,T]$. 
 In contrast to the minimax estimator $\hat{u}$, such estimators 
 are not only biased but also anticipating with respect to the Brownian filtration 
 $({\cal F}_t)_{t\in [0,T]}$. 
 This however poses no problem when one has access to complete paths 
 from time $0$ to $T$. 
\\ 

\noindent 
 For large values of $\sigma$ it can be shown that the percentage 
 gain of this estimator is at least equal to the universal constant 
\begin{equation} 
\label{gsint1}
\frac{16}{\pi^4} 
 \int_{\real^4} 
 e^{-\frac{x^2+y^2+z^2+r^2}{2}} 
 \frac{ dx dy dz dr}{x^2+9y^2+25z^2+49r^2} 
\end{equation} 
 which approximately represents $11.38\%$, see \eqref{gsint} below. 
\\ 
 
\noindent 
 We proceed as follows. 
 In Section~\ref{2} we use stochastic calculus in 
 the independent increment case to derive a Cramer-Rao bound over 
 all unbiased drift estimators. 
 This bound is attained by the process $\hat{u}:= (X_t)_{t\in [0,T]}$, 
 which will be considered as an efficient drift estimator. 
 In Section~\ref{2.0} we compute the Bayes estimators obtained 
 under prior Gaussian distributions. 
 We show that these Bayes estimators are admissible, and use 
 them to prove that the drift estimator $\hat{u}$ is minimax. 
 The tools and results presented in Sections~\ref{2} and \ref{2.0} 
 are not surprising, but we did not find any source covering them in the literature. 
 In Section~\ref{2.1} we recall the elements of
 analysis and integration by parts on Gaussian space 
 which will be needed in Section~\ref{3} 
 to construct superefficient drift estimators for Gaussian processes 
 using superharmonic random functionals on Gaussian space. 
 The superefficiency of these estimators will show, as in the 
 classical case, that the minimax estimator $\hat{u}$ is not admissible.  
 In Section~\ref{app} we give examples of nonnegative superharmonic
 functionals using cylindrical functionals and potential theory on Gaussian space. 
 Examples are considered in Section~\ref{4} in case $u$ is deterministic. 
 We show that the James-Stein estimators of Berger and Wolpert \cite{berger1} 
 can be recovered as particular cases in our approach, 
 and we provide numerical simulations for the gain of such estimators. 
 It turns out that in those examples,
 the gain obtained in comparison with the minimax 
 estimator $(X_t)_{t\in [0,T]}$ 
 is a function of $\sigma^2/T$, thus making $\sigma$ and $T$ play 
 inverse roles, unlike in the usual setting of Brownian rescaling. 
\\ 

\noindent 
 This paper is an extended version of \cite{pr3} and provides 
 proofs of the results presented in \cite{p-r-c}.  
\section*{Notation} 
 Let $T>0$. 
 Consider a real-valued centered Gaussian process $(X_t)_{t\in [0,T]}$ 
 with covariance function 
$$ 
 \gamma (s,t) = \E [ X_s X_t ], \qquad s,t\in [0,T], 
$$ 
 on a probability space $(\Omega , {\cal F} , \P )$, 
 where ${\cal F}$ is the $\sigma$-algebra generated by $X$. 
 Recall that $(X_t)_{t\in [0,T]}$ can be represented in different ways 
 as an isonormal Gaussian process 
 on a real separable Hilbert space $H$, i.e. as an isometry 
 $X : H \to L^2 (\Omega , {\cal F} , P)$ such that 
 $\{ X (h) \ : \ h\in H\}$ is a family of centered 
 Gaussian random variables satisfying 
$$\E 
 [ X (h) X (g) ] 
 = \langle h , g \rangle_H, 
 \qquad
 h,g \in H 
, 
$$ 
 where $\langle \cdot , \cdot \rangle_H$ and $\Vert \cdot \Vert_H$
 denote the scalar product and norm on $H$. 
\\ 
 
\noindent 
 One can distinguish two main types of such isonormal representations 
 of $X_t$, see e.g. \cite{amn} and \cite{berger1} respectively for details. 
\begin{description} 
\item{\bf (A)} 
 Paley-Wiener expansions. 
 In this case, $H$ is the completion 
 of the linear space generated by the functions 
 $\stackrel{}{\chi}_t \! \! (s) =\min (s,t)$, $s,t\in [0,T]$, 
 with respect to the norm  
$$ 
 \langle 
 \stackrel{}{\chi}_t 
 , 
 \stackrel{}{\chi}_s \rangle_H : = \gamma (s,t)
, 
 \qquad 
 s, t \in [0,T] 
, 
$$ 
 and $ X (\cdot )$ is constructed on $H$ from 
 $X ( \stackrel{}{\chi}_t ) : = X_t$, $t\in [0,T]$, 
 i.e. we have 
$$ 
 X ( \stackrel{}{\chi}_t ) 
 = 
 \sum_{k=0}^\infty 
 \langle 
 \stackrel{}{\chi}_t 
 , 
 h_k (t) 
 \rangle_H 
 X (h_k) 
, \qquad 
 t\in [0,T], 
$$ 
 for any orthonormal basis $(h_k)_{k\in\inte}$ of $H$. 
 Assume in addition $\gamma (s,t)$ has the form 
$$ 
 \gamma (s,t) = \int_0^{s\wedge t} K(t,r)K(s,r) dr 
, 
 \qquad 
 s,t\in [0,T] 
, 
$$ 
 where $K(\cdot , \cdot )$ is a deterministic kernel and 
$$ 
 (Kh)(t) 
 : = 
 \int_0^t K(t,s) \dot{h} (s) ds 
$$ 
 is differentiable in $t\in [0,T]$, 
 and let $K^*$ denote the adjoint of $K$ with respect to 
$$ 
 \langle h , g \rangle 
 : = \langle \dot{h} , \dot{g} \rangle_{L^2([0,T] , dt )}. 
$$ 
 The scalar product in $H$ then satisfies 
$$ 
 \langle h , g \rangle_H 
 = \langle K^* h , K^* g \rangle 
 = \langle h , \Gamma g \rangle 
, 
$$ 
 where $\Gamma = K K^*$, and we have the decomposition 
$$ 
 X_t 
 = 
 \sum_{k=0}^\infty 
 \langle 
 \stackrel{}{\chi}_t 
 , 
 h_k 
 \rangle_H 
 X (h_k) 
 = \sum_{k=0}^\infty 
 \langle 1_{[0,t]} 
 , 
 \dot{\Gamma} h_k 
 \rangle_{L^2([0,T] , dt )} 
 X (h_k) 
 = \sum_{k=0}^\infty 
 \Gamma h_k (t) 
 X (h_k) 
,
$$ 
 $t\in [0,T]$. 
 In this case we also have the representation 
$$ 
 X_t 
 = \int_0^t K(t,s) dW_s, 
 \qquad 
 t\in [0,T], 
$$ 
 where $(W_s)_{s\in [0,T]}$ is a standard Brownian 
 motion, cf. \cite{amn}. 
\item{\bf (B)} 
 Karhunen-Lo\`eve expansions. This framework is used in \cite{berger1}. 
 In this case, $\mu$ is a finite Borel measure on $[0,T]$ and 
 $H$ is defined from 
$$ 
 \langle 
 h 
 , 
 g 
 \rangle_H 
 = 
 \langle 
 h 
 , 
 \Gamma 
 g 
 \rangle 
, 
$$ 
 where 
$$\langle h , g \rangle 
 : = 
 \langle h , g \rangle_{L^2([0,T],d\mu )} 
, 
$$ 
 and  
$$ 
 ( \Gamma g ) (t) 
 = 
 \int_0^T g(s) \gamma (s,t) \mu (ds ), 
\qquad 
 t\in [0,T] 
, 
$$ 
 with 
$$ 
 X (h) = \int_0^T X_s h(s) \mu (ds), \qquad h\in H 
. 
$$ 
 Given $(h_k)_{k\in\inte}$ an orthonormal basis of 
 $L^2([0,T],d\mu)$, we have the expansion 
$$X_t = \sum_{k=0}^\infty h_k (t) X (h_k) 
, 
 \qquad 
 t\in [0,T] 
. 
$$ 
\end{description} 
 In the sequel we will use mainly the framework $(A)$ with 
 $\langle h , g \rangle = \langle \dot{h} , \dot{g} \rangle_{L^2([0,T] , dt )}$, 
 which is  
 better adapted to our approach, although some results valid 
 in the general framework of Gaussian processes will be valid 
 for $(B)$ as well. 
\\ 

\noindent 
 The Girsanov theorem for Gaussian processes, 
 cf. e.g. \cite{nualartm2}, states that $X^u$ defined as 
$$ 
 X^u(g) : = X(g) - \langle g , u \rangle 
 = X(g) - \langle g , \Gamma^{-1} u \rangle_H, 
 \qquad g\in H, 
$$ 
 where $u\in H$ is deterministic, 
 has same law as $X$ under the probability $\P_u$ defined by 
$$ 
 \frac{d\P_u}{d\P} 
 = \exp \left( X( \Gamma^{-1} u ) - \frac{1}{2} 
 \Vert \Gamma^{-1} u \Vert_H^2 
 \right) 
. 
$$ 
 In other terms, in case $(A)$ we have 
\begin{eqnarray*} 
 X_t - u(t) 
 & = & 
 X_t - \Gamma \Gamma^{-1} u(t) 
\\ 
 & = & 
 \sum_{k=0}^\infty 
 \Gamma h_k (t) 
 ( X(h_k) - \langle h_k, \Gamma^{-1} u\rangle_H ) 
\\ 
 & = & 
 \sum_{k=0}^\infty 
 \Gamma h_k (t) X^u (h_k) 
, \qquad 
 t\in [0,T] 
, 
\end{eqnarray*} 
 where $(h_k)_{k\in\inte}$ is orthonormal basis of $H$, 
 and in case $(B)$, 
\begin{eqnarray*} 
 X_t - u(t) 
 & = & 
 \sum_{k=0}^\infty 
 h_k (t) 
 ( 
 X( h_k) - \langle h_k, u \rangle_{L^2([0,T],d\mu)} 
 ) 
\\ 
 & = & 
 \sum_{k=0}^\infty 
 h_k (t) 
 ( 
 X(h_k) - \langle h_k, \Gamma^{-1} u \rangle_H 
 ) 
\\ 
 & = & 
 \sum_{k=0}^\infty 
 h_k (t) X^u (h_k) 
, \qquad 
 t\in [0,T] 
, 
\end{eqnarray*} 
 where $(h_k)_{k\in\inte}$ is orthonormal basis of 
 $L^2([0,T],d\mu)$. 
\section{Efficient drift estimator} 
\label{2} 
 Here we work in the framework of $(A)$, in the particular case where 
 $(X_t)_{t\in [0,T]}$ has independent increments, 
 i.e. 
$$\gamma (s,t) = \int_0^{s\wedge t} \sigma^2_u du 
, 
$$ 
 where $\sigma \in L^2([0,T] , dt )$ is an a.e. non-vanishing function, 
$$(\dot{K} h)(t) = (\dot{K}^* h )(t) = \sigma_t \dot{h} (t), \qquad 
 t\in [0,T], 
$$ 
 with $K(t,r) = 1_{[0,t]} (r) \sigma_r$ and 
$$ 
 \Gamma h (t) = \int_0^t \dot{h}_s \sigma_s^2 ds, 
 \qquad 
 t\in [0,T]. 
$$ 
 In other terms, $(X_t )_{t \in [0,T]}$ is a continuous Gaussian martingale 
 with 
 quadratic variation $\sigma^2_t dt$, 
 which can be represented as the time change 
$$X_t = W_{\int_0^t \sigma^2_s ds}, \qquad 
 t\in [0,T], 
$$ 
 of the standard Brownian motion $(W_t)_{t\in \real_+}$, 
 or as the stochastic integral process $\displaystyle 
 X_t = \int_0^t \sigma_s dW_s$, $t\in [0,T]$, and we have 
 $\displaystyle X (h) = \int_0^T \dot{h}(s) dX_s$, 
 $h \in H$, where 
$$H 
 = \left\{ 
 v : [0,T] \to \real
 \ : \ 
 v (t) = \int_0^t \dot{v}(s) ds,
 \ t\in [0,T], 
 \ 
 \dot{v} \in L^2 ([0,T], \sigma^2_t dt)
 \right\}$$ 
 is the Cameron-Martin space with inner product 
$$ 
 \langle v_1 , v_2 \rangle_H 
 = 
 \int_0^T \dot{v}_1 (s) \dot{v}_2 (s) \sigma^2_s ds, 
 \qquad v_1, v_2 \in H 
. 
$$ 
 
\noindent 
 Let $({\cal F}_t)_{t\in [0,T]}$ denote the filtration generated 
 by $(X_t)_{t\in [0,T]}$, and for $u$ an ${\cal F}_t$-adapted process, 
 let $\P^\sigma_u$ denote the translation of the Wiener measure 
 on $\Omega$ by $u$, i.e. $\P^\sigma_u$ is the measure on $\Omega$ under which 
$$X^u_t : = X_t - u_t 
 , \qquad t \in [0,T],
$$ 
 is a continuous Gaussian martingale with quadratic variation 
$$d\langle X^u , X^u \rangle_t = \sigma^2_t dt. 
$$ 
 Consider $u$ an ${\cal F}_t$-adapted processes of the form 
$$ 
 u_t = \int_0^t \dot{u}_s ds, \qquad t\in [0,T] 
, 
$$ 
 with 
$$ 
 \E^\sigma 
 \left[ 
 \int_0^T \frac{\dot{u}^2_s}{\sigma^2_s} ds 
 \right] 
 < \infty 
. 
$$ 
 By the Girsanov theorem, $\P^\sigma_u$ is absolutely continuous with 
 respect to $\P^\sigma$, with 
$$d\P^\sigma_u = \Lambda ( u ) d\P^\sigma, 
$$ 
 where 
$$ 
 \Lambda ( u )
 : = 
 \exp \left( 
 \int_0^T 
 \frac{\dot{u}_s}{\sigma^2_s} 
 dX_s 
 - 
 \frac{1}{2} \int_0^T \frac{\dot{u}_s^2}{\sigma^2_s} ds \right)
$$ 
 denotes the Girsanov-Cameron-Martin density, 
 the canonical process 
 $(X_t)_{t\in [0,T]}$ becomes a continuous Gaussian semimartingale 
 under $\P^\sigma_u$, 
 with quadratic variation $\sigma^2_t dt$ and  drift $\dot{u}_t dt$. 
 The expectation under $\P_u$ will be denoted by $\E_u$. 
\begin{definition} 
 A drift estimator $\xi$ is called unbiased if 
$$ 
 \E_u [\xi_t]= \E_u [u_t], \quad t \in [0,T], 
$$ 
 for all square-integrable ${\cal F}_t$-adapted process 
 $(u_t)_{t\in [0,T]}$. 
 It is called adapted if the process $(\xi_t)_{t\in [0,T]}$ is 
 ${\cal F}_t$-adapted. 
\end{definition} 
\noindent 
 Here, the canonical process $(X_t)_{t\in [0,T]}$ will be considered 
 as an unbiased estimator of own its drift $(u_t)_{t\in [0,T]}$ 
 under $\P^\sigma_u$, with risk defined as 
$$ 
 \E^\sigma_u \left[
 \Vert X - u \Vert_{L^2([0,T] , d\mu )}^2
 \right]
 = 
 \int_0^T \E^\sigma_u 
 \left[
 | X^u_t |^2 \right]
 \mu ( dt ) 
 = 
 \int_0^T \int_0^t \sigma^2_s ds 
 \mu ( dt ) 
, 
$$ 
 where $\mu$ is a finite Borel measure on $[0,T]$. 
 Clearly this estimator is consistent as $\sigma$ or $T$ tend to $0$: 
 precisely, given $N$ independent samples 
$$(X_t^1 )_{t\in [0,T]}, \ldots , (X_t^N )_{t\in [0,T]}, 
$$ 
 of $(X_t)_{t\in [0,T]}$, the process 
\begin{equation} 
\label{xbar} 
 \bar{X}_t 
 : = \frac{ 
 X_t^1 + \cdots + X_t^N 
 }{N} 
, 
 \qquad  
 t\in [0,T],  
\end{equation} 
 is an unbiased estimator of $(u_t)_{t\in [0,T]}$ whose risk 
$$ 
 \E^{\sigma/\sqrt{N}}_u 
 \left[
 \Vert \bar{X} - u \Vert_{L^2([0,T] , d\mu )}^2
 \right] 
 = 
 \frac{1}{N} 
 \int_0^T \int_0^t \sigma^2_s ds  \mu ( dt ) 
$$ 
 converges to zero as $N$ goes to infinity. 
\\ 

\noindent 
 The justification of the use of $\hat{u}=(X_t)_{t\in [0,T]}$ as an 
 efficient estimator 
 comes from the following proposition which 
 allows us to compute a Cramer-Rao bound attained by $\hat{u}$. 
 Here the parameter space is restricted to the space 
 of adapted processes in $L^2 (\Omega \times [0,T] , \P \otimes \mu )$, 
 which corresponds in a sense to a parametric estimation. 
\begin{prop} 
\label{prop:CramerRaoBound} 
 {Cramer-Rao inequality.} 
 For any unbiased and adapted estimator $\xi$ of 
 $u$ we have 
\begin{equation} 
\label{11}
 \E^\sigma_u \left[ \int_0^T |\xi_t - u_t|^2  \mu ( dt ) \right] 
 \geq 
 \crb (\sigma,\mu, \hat{u} )  
,
\end{equation}
 where $u \in L^2 (\Omega \times [0,T] , \P_u^\sigma \otimes \mu )$ 
 is adapted and the Cramer-Rao type bound 
$$ 
 \crb (\sigma,\mu, \hat{u} )  
 := \int_0^T \int_0^t \sigma^2_s ds \mu ( dt ) 
$$ 
 is independent of $u$ and attained by the 
 efficient estimator $\hat{u}=X$.
\end{prop} 
\begin{Proof} 
 Since $\xi$ is unbiased, for all $\zeta\in H$ we have 
\begin{eqnarray*} 
 \E^\sigma_{u+\varepsilon \zeta}[\xi_t] 
 & = & 
 \E^\sigma_{u+\varepsilon \zeta}[u_t 
 + 
 \varepsilon 
 \zeta_t ] 
\\ 
 & = & 
 \E^\sigma_{u+\varepsilon \zeta}[u_t] 
 + 
 \varepsilon 
 \E^\sigma_{u+\varepsilon \zeta}[\zeta_t ] 
\\ 
 & = & 
 \E^\sigma_{u+\varepsilon \zeta}[u_t] +\varepsilon \zeta_t,
 \quad t\in [0,T], \quad \varepsilon \in \real, 
\end{eqnarray*} 
 hence 
\begin{eqnarray*}
 \zeta_t
 &=& \frac{d}{d\varepsilon} \E^\sigma_{u+\varepsilon \zeta}[
 \xi_t - u_t
 ]_{|\varepsilon=0}
\\
&=& \frac{d}{d\varepsilon} \E^\sigma 
 [
 ( \xi_t - u_t )
 \Lambda ( u+\varepsilon \zeta ) ]_{| \varepsilon=0}
\\
 &=& 
 \E^\sigma 
 \left[
 (
 \xi_t - u_t
 ) 
 \frac{d}{d\varepsilon}
 \Lambda ( u+\varepsilon \zeta )_{| \varepsilon=0} 
 \right]
\\
&=&
 \E^\sigma_u \left[
 ( \xi_t - u_t )
 \frac{d}{d\varepsilon}
 \log \Lambda ( u+\varepsilon \zeta )_{| \varepsilon=0}
\right]
\\
 &=& 
 \E^\sigma_u \left[
 (
 \xi_t - u_t
 )
 \left( 
 \int_0^T \frac{\dot{\zeta}_s}{\sigma^2_s } 
 dX_s
 - 
 \int_0^T \frac{\dot{\zeta}_s\dot{u}_s}{\sigma^2_s} 
 ds 
 \right) 
 \right]
\\
 &=&
 \E^\sigma_u \left[
 (
 \xi_t - u_t
 )
 \int_0^T 
 \frac{\dot{\zeta}_s}{\sigma^2_s} 
 dX^u_s
 \right]
\\
 &=&
 \E^\sigma_u \left[
 (
 \xi_t - u_t
 )
 \int_0^t 
 \frac{\dot{\zeta}_s}{\sigma^2_s} 
 dX^u_s
 \right]
, 
\end{eqnarray*}
 where the exchange between expectation and derivative is justified by 
 classical uniform integrability arguments. 
 Thus, by the Cauchy-Schwarz inequality and the It\^o isometry we have 
$$ 
 \zeta_t^2
 \leq  
 \E^\sigma_u \left[ \left( 
 \int_0^t 
 \frac{\dot{\zeta}_s}{\sigma^2_s} 
 dX^u_s \right)^2 \right] \E^\sigma_u [ | \xi_t - u_t|^2 ]
 = 
 \int_0^t 
 \frac{\dot{\zeta}_s^2}{\sigma^2_s} 
  ds \E^\sigma_u [ | \xi_t - u_t|^2 ], 
 \qquad 
 t\in [0,T]. 
$$ 
 It then suffices to take 
$$\zeta_t = \int_0^t 
 \sigma^2_s ds, 
 \qquad 
 t\in [0,T], 
$$ 
 to get
\begin{equation} 
\label{eqv}
 \var^\sigma_u [\xi_t]
 =
 \E^\sigma_u [ |\xi_t - u_t|^2 ] \geq 
 \int_0^t \sigma^2_s ds, \qquad t\in [0,T] 
,
\end{equation}
 which leads to \eqref{11} after integration with respect to 
 $\mu (dt)$. 
 As noted above, $\hat{u} = (X_t)_{t\in [0,T]}$ 
 is clearly unbiased under $\P^\sigma_u$ and it attains 
 the lower bound $\crb (\sigma,\mu, \hat{u} ) $. 
\end{Proof} 
\noindent 
 Recall that the classical linear 
 parametric estimation problem for the drift of a diffusion 
 consists in estimating the coefficient $\theta$ appearing 
 in 
$$d\xi_t = \theta a_t(\xi_t) dt + dY_t, \qquad 
 \xi_0 = 0 
, 
$$ 
 with a maximum likelihood estimator $\hat{\theta}_T$ given by 
\begin{equation} 
\label{mlet} 
\hat{\theta}_T = \frac{\int_0^T a_t (\xi_t)d\xi_t}{\int_0^T a_t^2(\xi_t)dt} 
, 
\end{equation} 
 cf. \cite{liptser}, \cite{prakasarao} for Brownian motion and 
 \cite{tudorviens} for an extension to fractional Brownian motions. 
\\ 

\noindent 
 Here we consider the nonparametric functional estimation of the drift of 
 a one-dimensional drifted Brownian motion 
 $(X_t)_{t\in \real_+}$ with decomposition 
\begin{equation} 
\label{jk} 
 dX_t = \dot{u}_t dt + dX^u_t, 
\end{equation} 
 where $(\dot{u}_t)_{t\in [0,T]} \in L^2 (\Omega \times [0,T])$ is 
 an adapted process and 
 $(X^u_t)_{t\in \real_+}$ is a standard Brownian motion with 
 quadratic variation $\sigma^2_t$ 
 under a probability $\P^\sigma_u$. 
\\ 

\noindent 
 In case $u$ is constrained to have the form $u_t = \theta t$, $t\in [0,T]$, 
 $\theta \in \real$, our efficient estimator 
 $\hat{u}$ satisfies $\hat{u}_T = \hat{\theta}_T T$, 
 where $\hat{\theta}_T$ is given by \eqref{mlet}, $T>0$, with 
 the asymptotics $\hat{\theta}_T\to \theta$ in probability 
 as $T$ tends to infinity. 
 The asymptotics is not in large time since $T$ can be a fixed parameter, 
 but the efficient estimator 
 $\hat{u} = X$ converges to $u$ as $\sigma$ tends to $0$, 
 or equivalently as $T$ tends to $0$ by rescaling. 
\\ 

\noindent 
 To close this section we note that, at least informally, 
 $\hat{u} = (X_t)_{t\in [0,T]}$ can be viewed as a 
 maximum likelihood estimator of its own adapted drift $(u_t)_{t\in [0,T]}$ 
 under $\P_u^\sigma$. 
 Indeed the functional differentiation of the Cameron-Martin 
 density 
$$\frac{d}{d\varepsilon} \Lambda ( \hat{u}+\varepsilon \zeta )_{\mid \varepsilon = 0} = 0
, \qquad \zeta \in H ,
$$ 
 implies 
$$\int_0^T \frac{\dot{\zeta}_s}{\sigma^2_s } d X_s  
 - 
 \int_0^T \frac{\dot{\zeta}_s}{\sigma^2_s } d\hat{u}_s 
 = 0 
, \qquad 
 \zeta \in H , 
$$ 
 which leads to $X = \hat{u}$. 
\section{Bayes estimators} 
\label{2.0} 
\noindent 
 In this section we 
 consider Bayes estimators which will be useful in proving 
 the minimaxity of the estimator $\hat{u} = X$ in the framework of $(A)$ for Gaussian 
 processes with non-necessarily independent increments. 
 We will make use of the next lemma which is classical in the framework of 
 Gaussian filtering and is proved in the Appendix. 
\begin{lemma} 
\label{ljk} 
 Let $Z$ be a Gaussian process with covariance operator $\Gamma_\tau$ 
 and drift $v\in H$, 
 and assume that $X$ is a Gaussian process with drift $Z$ 
 and covariance operator $\Gamma$ given $Z$. 
 Then, conditionally to $X$, $Z$ has drift 
$$ 
 f 
 \mapsto 
 \langle 
 f 
 , 
 ( \Gamma+\Gamma_\tau)^{-1}\Gamma 
 v 
 \rangle  
 + 
 X ( 
 ( \Gamma+\Gamma_\tau)^{-1}\Gamma_\tau 
 f 
 ) 
 \quad 
 \mbox{and covariance} 
 \quad 
 \Gamma_\tau ( \Gamma + \Gamma_\tau )^{-1} \Gamma 
. 
$$ 
\end{lemma} 

\noindent 
 Note that unlike in Proposition~\ref{prop:CramerRaoBound}, 
 no adaptedness or unbiasedness restriction is made on 
 $\xi$ in the infimum taken in \eqref{infr} below. 
\begin{prop} 
\label{best} 
 {Bayes estimator.} 
 Let $P^\tau_v$ denote the Gaussian distribution on $\Omega$ 
 with covariance operator $\Gamma_\tau$ and drift $v\in H$. 
 The Bayes risk 
\begin{equation} 
\label{br} 
 \int_\Omega 
 \E_z \left[ \int_0^T | \xi_t - z_t|^2  \mu ( dt ) 
 \right] 
 d\P^\tau_v (z) 
\end{equation} 
 of any estimator $(\xi_t)_{t\in [0,T]}$ 
 on $\Omega$ under the prior distribution $\P^{\tau}_v$ 
 is uniquely minimized by 
$$ 
 \xi^{\tau,v}_t 
 : = 
 \langle 
 \stackrel{}{\chi}_t , 
 ( \Gamma_\tau+\Gamma )^{-1} \Gamma 
 v \rangle 
 + 
 X ( 
 ( \Gamma_\tau+\Gamma )^{-1} \Gamma_\tau 
 \stackrel{}{\chi}_t  
 ) 
, 
 \qquad
 t\in [0,T] 
, 
$$ 
 which has risk 
\begin{equation} 
\label{infr} 
 \int_0^T 
 \langle 
 \stackrel{}{\chi}_t  
 , 
 \Gamma 
 ( \Gamma_\tau+\Gamma )^{-1} \Gamma_\tau 
 \stackrel{}{\chi}_t  
 \rangle 
 \mu ( dt ) 
 = 
 \inf_\xi 
 \int_\Omega 
 \E_z \left[ \int_0^T | \xi_t - z_t|^2 
 \mu ( dt ) 
 \right] 
 d\P^\tau_v (z) 
. 
\end{equation} 
\end{prop} 
\begin{Proof} 
 Let $Z$ denote a Gaussian process 
 with drift $v\in H$ and covariance $\Gamma_\tau$. 
 Recall (cf. Lemma~\ref{ljk}) that 
 if $X$ has drift $Z$ and covariance $\Gamma$ 
 then, conditionally to $X$, $(Z_t)_{t\in [0,T]}$ has drift 
$$ 
 t \mapsto 
 \langle 
 \stackrel{}{\chi}_t  
 , 
 (\Gamma_\tau + \Gamma )^{-1} 
 \Gamma 
 v 
 \rangle 
 + 
 X ( 
 (\Gamma_\tau + \Gamma )^{-1} 
 \Gamma_\tau 
 \stackrel{}{\chi}_t  
 ) 
$$ 
 and covariance 
 $
 \Gamma_\tau 
 (\Gamma_\tau + \Gamma )^{-1} 
 \Gamma$. 
 Hence the Bayes risk of an estimator $\xi$ 
 under the prior distribution $\P^\tau_v$ is given by 
\begin{eqnarray*} 
\lefteqn{ 
 \int_\Omega 
 \E_z \left[ \int_0^T | \xi_t - z_t|^2 \mu ( dt ) 
 \right] 
 d\P^\tau_v (z) 
 = 
 \E \left[ 
 \E \left[ \int_0^T | \xi_t - Z_t|^2 \mu ( dt ) \Big| X 
 \right] 
 \right] 
} 
\\ 
 & = & 
 \E \left[ 
 \int_0^T | \xi_t - \E[Z_t \mid X ] |^2 \mu ( dt ) 
 \right] 
 + 
 \E \left[ 
 \int_0^T 
 \var ( Z_t | X ) 
 \mu ( dt ) 
 \right] 
\\ 
 & = & 
 \E \left[ 
 \int_0^T 
 \left| \xi_t - 
 \langle 
 \stackrel{}{\chi}_t , 
 ( \Gamma_\tau+\Gamma )^{-1} \Gamma 
 v \rangle 
 - 
 X ( 
 ( \Gamma_\tau+\Gamma )^{-1} \Gamma_\tau 
 \stackrel{}{\chi}_t 
 ) 
 \right|^2 
 \mu ( dt ) 
 \right] 
\\ 
 & & 
 + 
 \int_0^T 
 \langle 
 \stackrel{}{\chi}_t 
 , 
 \Gamma 
 ( \Gamma_\tau+\Gamma )^{-1} \Gamma_\tau 
 \stackrel{}{\chi}_t 
 \rangle 
 \mu ( dt ) 
, 
\end{eqnarray*} 
 which is minimized by 
$$ 
 \xi^{\tau,v}_t 
 : =  
 \E[Z_t \mid X ] 
 = 
 \langle 
 \stackrel{}{\chi}_t , 
 ( \Gamma_\tau+\Gamma )^{-1} \Gamma 
 v \rangle 
 - 
 X ( 
 ( \Gamma_\tau+\Gamma )^{-1} \Gamma_\tau 
 \stackrel{}{\chi}_t  
 ) 
, 
 \qquad t\in [0,T] 
. 
$$ 
\end{Proof} 
\noindent 
 Clearly $\xi^{\tau,v}$ is unique in the sense that it is the only 
 estimator to minimize the Bayes risk \eqref{br}. 
 This shows in particular that every $\xi^{\tau,v}$ is admissible 
 in the sense that if an estimator $\xi$ satisfies 
$$ 
 \E_z \left[
 \Vert \xi - z \Vert_{L^2 ([0,T] , d\mu )}^2
 \right] 
 \leq 
 \E_z \left[
 \Vert \xi^{\tau,v} - z \Vert_{L^2 ([0,T] , d\mu )}^2
 \right], \qquad 
 z\in \Omega 
, 
$$ 
 then 
\begin{eqnarray*} 
 \int_\Omega 
 \E_z \left[
 \Vert \xi - z \Vert_{L^2 ([0,T] , d\mu )}^2
 \right] 
 d\P^\tau_v 
 & \leq & 
 \int_\Omega 
 \E_z \left[
 \Vert \xi^{\tau,v} - z \Vert_{L^2 ([0,T] , d\mu )}^2
 \right] 
 d\P^\tau_v 
\\ 
 & = & 
 \int_0^T 
 \langle 
 \stackrel{}{\chi}_t  
 , 
 \Gamma 
 ( \Gamma_\tau+\Gamma )^{-1} \Gamma_\tau 
 \stackrel{}{\chi}_t  
 \rangle 
 \mu ( dt ) 
, 
\end{eqnarray*} 
 hence 
\begin{equation} 
\label{mrsk} 
 \int_\Omega 
 \E_z \left[
 \Vert \xi - z \Vert_{L^2 ([0,T] , d\mu )}^2
 \right] 
 d\P^\tau_v 
 = 
 \int_0^T 
 \langle 
 \stackrel{}{\chi}_t  
 , 
 \Gamma 
 ( \Gamma_\tau+\Gamma )^{-1} \Gamma_\tau 
 \stackrel{}{\chi}_t  
 \rangle 
 \mu ( dt ) 
, 
\end{equation} 
 and $\xi = \xi^{\tau,v}$ by Proposition~\ref{best}. 
\\ 

\noindent 
 The Bayes estimator $\xi^{\tau , v}$ is biased in general, 
 and for deterministic $u\in H$ its mean square error under 
 $\P_u$ is equal to 
\begin{eqnarray} 
\label{abid} 
\lefteqn{ 
 \E_u \left[ \int_0^T | \xi^{\tau , v}_t - u_t|^2  \mu ( dt ) 
 \right] 
} 
\\ 
\nonumber 
 & = & 
 \E_u \left[ \int_0^T \Big| 
 \xi^{\tau , v}_t - \E_u [ \xi^{\tau , v}_t ] 
 \Big|^2  \mu ( dt ) 
 \right] 
 + 
 \E_u \left[ 
 \int_0^T | 
 \E_u [ \xi^{\tau , v}_t ] 
 - u_t 
 |^2 
 \mu ( dt ) 
 \right] 
\\ 
\nonumber 
 & = & 
 \E_u \left[ \int_0^T \Big| 
 X^u ( 
 ( \Gamma_\tau+\Gamma )^{-1} \Gamma_\tau 
 \stackrel{}{\chi}_t  
 ) 
 \Big|^2  \mu ( dt ) 
 \right] 
 + 
 \int_0^T \Big| 
 \langle 
 \stackrel{}{\chi}_t , 
 ( \Gamma_\tau+\Gamma )^{-1} \Gamma 
 ( v - u ) 
 \rangle 
 \Big|^2 
 \mu ( dt ) 
\\ 
\nonumber 
 & = & 
 \int_0^T 
 \langle 
 ( \Gamma_\tau+\Gamma )^{-1} 
 \Gamma_\tau 
 \stackrel{}{\chi}_t 
 , 
 \Gamma ( \Gamma_\tau+\Gamma )^{-1} 
 \Gamma_\tau 
 \stackrel{}{\chi}_t 
 \rangle 
 \mu ( dt ) 
\\ 
\nonumber 
& &
 + 
 \int_0^T \Big| 
 \langle 
 \stackrel{}{\chi}_t , 
 ( \Gamma_\tau+\Gamma )^{-1} \Gamma 
 ( v - u ) 
 \rangle 
 \Big|^2 
 \mu ( dt ) 
, 
\end{eqnarray} 
 which shows that 
$$ 
 \sup_{u\in H} 
 \E_u \left[ \int_0^T | \xi^{\tau , v}_t - u_t|^2 \mu ( dt ) 
 \right] 
 = 
 + \infty 
, 
$$ 
 hence $\xi^{\tau , v}$ is not minimax. 
\\ 
 
\noindent In the independent increment case of Section~\ref{2} we have, 
 if 
 $\Gamma_\tau f(s) = \tau_s^2 f(s)$, $s\in [0,T]$: 
$$ 
 \xi^{\tau,v}_t 
 : = 
 \int_0^t \frac{\sigma^2_s\dot{v}_s}{\tau^2_s+\sigma^2_s} ds 
 + 
 \int_0^t \frac{\tau^2_s}{\tau^2_s+\sigma^2_s} dX_s 
, 
 \qquad
 t\in [0,T] 
, 
$$ 
 with risk 
\begin{equation} 
\label{infr2} 
 \int_0^T 
 \int_0^t \frac{\tau^2_s\sigma^2_s}{\tau^2_s+\sigma^2_s} ds 
 \mu ( dt ) 
 = 
 \inf_\xi 
 \int_\Omega 
 \E_z \left[ \int_0^T | \xi_t - z_t|^2 
 \mu ( dt ) 
 \right] 
 d\P^\tau_v (z) 
. 
\end{equation} 
\noindent 
 Assuming now that $\Gamma_\tau f (t) = \tau^2 f(t)$, $t\in [0,T]$, 
 the Bayes risk 
$$ 
 \int_0^T 
 \langle 
 \stackrel{}{\chi}_t , 
 \Gamma_\tau  
 ( \Gamma_\tau+\Gamma )^{-1} 
 \Gamma 
 \stackrel{}{\chi}_t 
 \rangle 
 \mu ( dt ) 
 = 
 \int_0^T 
 \langle 
 \stackrel{}{\chi}_t , 
 ( I + \Gamma /\tau^2 )^{-1} 
 \Gamma \stackrel{}{\chi}_t 
 \rangle 
 \mu ( dt ) 
, 
$$ 
 of $\xi^{\tau , v}$, $\tau \in \real$, 
 converges as $\tau \to \infty$ to the bound 
$$ 
 \crb (\sigma,\mu, \hat{u} ) 
 = 
 \int_0^T 
 \langle 
 \stackrel{}{\chi}_t , 
 \Gamma 
 \stackrel{}{\chi}_t 
 \rangle 
 \mu ( dt ) 
,
$$ 
 hence it follows in the next proposition that, as in 
 the finite dimensional Gaussian case, the 
 estimator $\hat{u} = (X_t)_{t\in [0,T]}$ is minimax. 
 Note again that unlike in Proposition~\ref{prop:CramerRaoBound}, 
 no adaptedness condition is imposed on $\xi$ in the infima 
 \eqref{infr} and \eqref{nmx}. 
\begin{prop} 
\label{minimax} 
 The estimator $\hat{u}=X$ is minimax. 
 For all $u\in \Omega$ we have 
\begin{equation} 
\label{nmx} 
 \crb ( \gamma ,\mu, \hat{u} ) 
 = 
 \E_u \left[ \int_0^T | X_t - u_t|^2 \mu ( dt ) \right] 
 = \inf_\xi 
 \sup_{v \in \Omega } 
 \E_v \left[ \int_0^T | \xi_t - v_t|^2 \mu ( dt ) \right] 
. 
\end{equation} 
\end{prop} 
\begin{Proof} 
 Clearly, taking $\xi=0$ yields 
$$ 
 \crb (\gamma ,\mu, \hat{u} ) 
 = 
 \sup_{u\in \Omega} 
 \E_u \left[ \int_0^T | X_t - u_t|^2 \mu ( dt ) \right] 
 \geq \inf_\xi 
 \sup_{u\in \Omega} 
 \E_u \left[ \int_0^T | \xi_t - u_t|^2 \mu ( dt ) \right] 
. 
$$ 
 On the other hand, from Proposition~\ref{best}, for all processes 
 $\xi$ we have 
\begin{eqnarray*} 
 \sup_{u \in \Omega} 
 \E_u \left[ \int_0^T | \xi_t - u_t|^2 \mu ( dt ) \right] 
 & \geq & 
 \int_\Omega 
 \E_z \left[ \int_0^T | \xi_t - z_t|^2 \mu ( dt ) 
 \right] 
 d\P^\tau_0 (z)  
\\ 
 & \geq & 
 \int_0^T 
 \langle 
 \stackrel{}{\chi}_t , 
 ( I + \Gamma /\tau^2 )^{-1} 
 \Gamma 
 \stackrel{}{\chi}_t 
 \rangle 
 \mu ( dt ) 
, 
\end{eqnarray*} 
 for all $\tau > 0$, hence 
$$ 
 \inf_\xi 
 \sup_{u\in H} 
 \E_u \left[ \int_0^T | \xi_t - u_t|^2 \mu ( dt ) \right] 
 \geq 
 \int_0^T 
 \langle 
 \stackrel{}{\chi}_t , 
 \Gamma 
 \stackrel{}{\chi}_t 
 \rangle 
 \mu ( dt ) 
 = 
 \crb ( \gamma ,\mu, \hat{u} ) 
. 
$$ 
\end{Proof} 
\section{Malliavin calculus on Gaussian space}
\label{2.1} 
 Before proceeding to the construction 
 of Stein type estimators, we need to introduce 
 some elements of analysis on Gaussian space, see e.g. \cite{nualartm2}. 
 This construction is valid in both frameworks $(A)$ and $(B)$. 
 Given $u \in H$, let 
$$ 
 X^u = X - u. 
$$  

\noindent 
 We fix $(h_n)_{n \geq 1}$ a total subset of $H$ 
 and let ${\cal S}$ denote the space of cylindrical functionals of the 
 form 
\begin{equation}
\label{e1}
 F= f_n 
 \left(
 X^u ( h_1 ) 
 ,
 \ldots
 , 
 X^u ( h_n) 
 \right)
,
\end{equation}
 where $f_n$ is in the space of infinitely differentiable rapidly decreasing
 functions on $\real^n$, $n\geq 1$.
\noindent 
\begin{definition} 
 The $H $-valued Malliavin derivative is defined as
$$ 
 \nabla_t F =
 \sum_{i=1}^{n} 
 h_i ( t ) 
 \partial_i f_n 
 \left(
 X^u (h_1 ) 
 ,
 \ldots
 ,
 X^u (h_n) 
 \right)
,
$$
 for $F \in {\cal S}$ of the form \eqref{e1}. 
\end{definition} 
\noindent 
 It is known that $\nabla$ is closable, cf. Proposition~1.2.1 of 
 \cite{nualartm2}, and its closed domain will be denoted by $\Dom (\nabla )$. 
\begin{definition} 
 Let $D$ be defined on $\Dom ( \nabla )$ as 
$$D_t F : = ( \Gamma \nabla F ) (t) 
, \quad t\in [0,T], \quad F\in \Dom (\nabla ) 
. 
$$ 
\end{definition} 
\noindent 
 Let $\delta : L^2_u (\Omega ; H ) \to L^2 (\Omega , \P_u )$ 
 denote the closable adjoint of $\nabla$, i.e. the 
 divergence operator under $\P_u$, 
 which satisfies the integration by parts formula
\begin{equation} 
\label{intparpartiesvar1/N} 
 \E_u [ F \delta ( v ) ] 
 = \E_u [ \langle v , \nabla F \rangle_{H} ] 
, 
 \qquad 
 F \in \Dom ( \nabla ) 
 , 
 \quad 
 v \in \Dom (\delta )
, 
\end{equation} 
 with the relation 
$$\delta (h F ) = F X (h) - \langle h , \nabla F\rangle_H, 
$$ 
 cf. \cite{nualartm2}, for $F\in \Dom (\nabla )$ and $h\in H$ such 
 that $hF\in \Dom (\delta )$. 
 Note that \eqref{intparpartiesvar1/N} is an infinite-dimensional version 
 of the integration by parts \eqref{ibp2}, which can be proved e.g. using 
 the countable Gaussian random variables constructed from $X$. 
\begin{lemma} 
\label{prt} 
 We have 
$$ 
 \E_u [ F X^u_t ] 
 = \E_u [ 
 D_t F 
 ] 
, 
 \qquad 
 t\in [0,T] 
 , 
 \quad 
 F \in \Dom ( \nabla ) 
. 
$$ 
\end{lemma}
\begin{Proof} 
\begin{description} 
\item{(A)} 
 In the case of Paley-Wiener expansions we have 
\begin{eqnarray*} 
 \E_u [ F X^u_t ] 
 & = & 
 \E_u [ F X^u  ( \stackrel{}{\chi}_t ) ] 
\\ 
 & = & 
 \E_u [ F \delta ( \stackrel{}{\chi}_t ) ] 
\\ 
 & = & 
 \E_u [ \langle \stackrel{}{\chi}_t , \nabla F\rangle_H ] 
\\ 
 & = & 
 \E_u [ \langle 1_{[0,t]} , \dot{\Gamma} \nabla F\rangle_{L^2([0,T] , dt )} ] 
\\ 
 & = & 
 \E_u [ ( \Gamma \nabla F) ( t ) ] 
, 
 \qquad 
 F \in \Dom ( \nabla ) 
 , 
 \quad 
 t\in [0,T] 
. 
\end{eqnarray*} 
\item{(B)} 
 In the case of Karhunen-Lo\`eve expansions we have 
\begin{eqnarray*} 
 \E_u [ F X^u_t ] 
 & = & 
 \sum_{k=0}^\infty 
 h_k(t) 
 \E_u [ F X^u (h_k) ] 
\\ 
 & = & 
 \sum_{k=0}^\infty 
 h_k(t) 
 \E_u [ F \delta (h_k) ] 
\\ 
 & = & 
 \sum_{k=0}^\infty 
 h_k(t) 
 \E_u [ \langle h_k , \nabla F\rangle_H ] 
\\ 
 & = & 
 \sum_{k=0}^\infty 
 h_k(t) 
 \E_u [ \langle h_k , \Gamma \nabla F\rangle_{L^2([0,T],\mu )} ] 
\\ 
 & = & 
 \E_u [ 
 ( \Gamma \nabla F ) (t) 
 ] 
, 
 \qquad 
 F \in \Dom ( \nabla ) 
 , 
 \quad 
 t\in [0,T] 
. 
\end{eqnarray*} 
\end{description}  
\end{Proof} 
\noindent 
\begin{definition} 
 We define the Laplacian $\Delta$ by
$$ 
 \Delta F = \trace_{L^2([0,T] , d\mu )^{\otimes 2}} D D F 
 = \int_0^T D_t D_t F \mu (dt) 
$$ 
 on the space $\Dom (\Delta )$ made 
 of all $F\in \Dom (\nabla )$ 
 such that $D_tF \in \Dom (\nabla)$, $t\in [0,T]$, and 
 $(D_tD_tF)_{t\in [0,T]}\in L^2([0,T],\mu)$, $\P$-a.s. 
\end{definition} 
\noindent 
 If $F\in {\cal S}$ has the form \eqref{e1} we have 
$$ 
 \Delta F 
 = 
 \sum_{i,j=1}^n 
 \langle 
 \Gamma h_i 
 , 
 \Gamma h_j 
 \rangle_{L^2([0,T],\mu )} 
 \partial_i \partial_j f_n 
 \left(
 X^u ( h_1 ) 
 ,
 \ldots
 ,
 X^u ( h_n ) 
 \right)
. 
$$ 
 Unlike the Gross Laplacian $\Delta_G$ defined by 
$$\Delta_G F = \trace_{H^{\otimes 2}} \nabla \nabla F 
, 
$$ 
 the operator $\Delta$ is closable, as shown in the following
 proposition. 
\begin{prop}
\label{prop:closability}
 Closability of $\Delta$. 
 For any sequence $(F_n)_{n \in \inte}$ of random variables converging
 to $0$ in $L^2 (\Omega , \P_u )$ and such that
 $(\Delta F_n)_{n \in \inte}$ converges in $L^2 (\Omega , \P_u )$, we have
$$\lim_{n\to \infty }
 \Delta F_n = 0
.
$$
\end{prop}
\begin{proof} 
 Let $(G_n)_{n \in \inte}$ a sequence in $\mathcal{S}$ converging to $0$ in 
 $L^2(\Omega, \P_u )$, and such that $(\Delta G_n)_{n \in \inte}$ converges to $F$ 
 in $L^2(\Omega, \P_u )$. 
 For all $G\in\mathcal{S}$ we have, in the notation of $(A)$: 
\begin{eqnarray*}
\disp{\vert \langle \Delta G_n, G\rangle_{L^2 (\Omega , \P_u )} \vert} &=& \disp{ \left\vert 
 \E_u \left[ G \int_0^T D_tD_t G_n \mu(dt) \right] \right\vert } \\
&=& \disp{ \left\vert \int_0^T \E_u [\langle\nabla D_t G_n , 
 \stackrel{}{\chi}_t G \rangle_H] \, \mu(dt) 
 \right\vert } 
\\
&=& \disp{ \left\vert \int_0^T \E_u [D_t G_n \, \delta( 
 \stackrel{}{\chi}_t 
 G) ] \, \mu(dt)  \right\vert } \\
&=& \disp{ \left\vert \int_0^T \E_u [\langle \nabla G_n, 
 \stackrel{}{\chi}_t 
 \delta( 
 \stackrel{}{\chi}_t 
 G) \rangle_H ] \, \mu(dt)  \right\vert } \\
&=& \disp{ \left\vert \int_0^T \E_u [G_n \delta( 
 \stackrel{}{\chi}_t 
 \delta( 
 \stackrel{}{\chi}_t 
 G)) ] \, \mu(dt)  \right\vert } \\
&\leq& \disp{ \|G_n\|_{L^2 (\Omega, \P_u )} \int_0^T \|\delta( 
 \stackrel{}{\chi}_t 
 \delta( 
 \stackrel{}{\chi}_t 
 G))\|_{L^2(\Omega, \P_u )} \, \mu(dt), }
\end{eqnarray*} 
hence $\langle F, G \rangle_{L^2(\Omega, \P_u )} =0, \; G \in \mathcal{S}$, 
 which implies $F=0$.
\end{proof}
\noindent 
 We will say that a random variable $F$ in $\Dom (\Delta )$ 
 is $\Delta$-superharmonic on $\Omega$ if 
\begin{equation} 
\label{sph} 
\Delta F (\omega ) \leq 0, \qquad \P (d\omega)-a.s. 
\end{equation} 
\begin{remark} 
 In the independent increment case where $\gamma (s,t)$ is given by 
$$\gamma (s,t) = \int_0^{s\wedge t} 
 \sigma_u^2 dt 
, 
 \qquad 
 s,t\in [0,T] 
, 
$$ 
 we have 
\begin{equation} 
\label{intparpartiesvar1/N.1} 
\delta ( v ) = \int_0^T \dot{v}_t dX^u_t, 
\end{equation} 
 for every ${\cal F}_t$-adapted process $v \in L^2 (\Omega ; H , \P_u )$. 
\end{remark} 
\section{Superefficient drift estimators}
\label{3} 
 Our aim is to construct a superefficient estimator of $u$ of the form 
 $X+\xi$, whose mean square error is strictly smaller than the 
 minimax risk $\crb (\gamma , \mu , \hat{u})$ of Proposition~\ref{minimax} 
 when $\xi \in L^2 ([0,T]\times \Omega , \P_u \otimes \mu )$ 
 is a suitably chosen stochastic process. 
 This estimator will be biased and anticipating with respect 
 to the Brownian filtration. 
 In the next lemma we follow Stein's argument which 
 uses integration by parts but we replace \eqref{ibp2} by
 the duality relation \eqref{intparpartiesvar1/N} 
 between the gradient and divergence operators on Gaussian space. 
 The results of this section are valid in both frameworks $(A)$ and $(B)$. 
\begin{lemma} 
\label{lemma1} 
 {Unbiased risk estimate}. 
 For any $\xi \in L^2 (\Omega \times [0,T] , \P_u \otimes \mu )$ such that 
 $\xi_t \in \Dom (\nabla)$, $t\in [0,T]$, and 
 $(D_t\xi_t)_{t\in [0,T]}\in L^1 (\Omega \times [0,T] , \P_u \otimes \mu )$, we have 
\begin{equation} 
\label{eq:erreur}
 \E_u 
 \left[ 
 \Vert X + \xi - u \Vert_{L^2([0,T], \mu )}^2
 \right]
 = 
 \crb (\gamma ,\mu, \hat{u} ) 
 + 
 \Vert \xi \Vert_{L^2 (\Omega \times [0,T] , \P_u \otimes \mu )}^2
 + 2 
 \E_u \left[
 \int_0^T 
 D_t \xi_t 
 \mu ( dt ) 
 \right] 
. 
\end{equation}  
\end{lemma} 
\begin{Proof} 
 We have 
\begin{eqnarray*} 
\lefteqn{ 
 \E_u \left[ 
 \Vert X + \xi - u \Vert_{L^2([0,T] , d\mu )}^2 \right]
 = 
 { \E_u \left[ \int_0^T \Big{|} X^u_t + \xi_t \Big{|}^2 
 \mu ( dt ) 
 \right] }
} 
\\
&=& { \E_u \left[ \int_0^T | X^u_t |^2 
 \mu ( dt ) 
 \right] + 
 \Vert \xi \Vert_{L^2 ( \Omega \times [0,T] , \P_u \otimes \mu )}^2  +
 2 \E_u \left[ \int_0^T X^u_t \xi_t 
 \mu ( dt ) 
 \right] } \\
 &=& 
 \crb (\gamma ,\mu, \hat{u} ) 
 + 
 \Vert \xi \Vert_{L^2 ( \Omega \times [0,T] , \P_u \otimes \mu )}^2 
 + 2 \E_u \left[ \int_0^T X^u_t \xi_t 
 \mu ( dt ) 
 \right] 
, 
\end{eqnarray*} 
 and apply Lemma~\ref{prt} to obtain \eqref{eq:erreur}. 
\end{Proof} 
\noindent 
 The next proposition specializes the above lemma to 
 processes $\xi$ of the form 
$$\xi_t = D_t \log F, 
 \qquad t\in [0,T], 
$$ 
 where $F$ is an a.s. strictly positive and sufficiently 
 smooth random variable. 
\begin{prop} 
\label{prop:erreurcylindricalfunctions.1} 
 Logarithmic gradient. 
 Stein-type estimator. 
 For any $\P$-a.s. positive random variable 
 $F\in \Dom (\nabla)$ such that $D_t F\in \Dom (\nabla)$, $t\in [0,T]$, 
 and 
 $(D_tD_tF )_{t\in [0,T]}\in L^1 (\Omega \times [0,T] , \P_u \otimes \mu )$, we have 
$$ 
 \E_u 
 \left[ 
 \Vert X + 
 D \log F - u \Vert_{L^2 ([0,T] , d\mu )}^2
 \right] 
 = 
 \crb (\gamma ,\mu, \hat{u} ) 
 - 
 \E_u \left[
 \Vert 
 D \log F \Vert_{L^2 ([0,T], \mu )}^2 
 \right] 
 + 
 2 
 \E_u \left[
 \frac{ \Delta F }{F} 
 \right] 
.
$$ 
\end{prop}
\begin{Proof}
 From \eqref{eq:erreur} we have 
\begin{eqnarray*} 
\lefteqn{ 
 \! \! \! \! \! \! \! \! \! \! \! 
 \E_u 
 \left[ 
 \Vert X + D \log F - u \Vert_{L^2 ([0,T],\mu )}^2
 \right] 
} 
\\ 
 & = & 
 \crb (\gamma ,\mu, \hat{u} ) 
 + 
 \Vert D \log F \Vert_{L^2 (\Omega \times [0,T] , \P_u \otimes \mu )}^2
 + 
 2 \E_u \left[ \int_0^T D_t D_t \log F 
 \mu ( dt ) 
 \right] 
\\ 
& = &
 \crb (\gamma ,\mu, \hat{u} ) 
 + 
 \E_u \left[
 \int_0^T \left(
 \bigg| \frac{D_t F }{F} \bigg|^2
 + 2 D_t D_t \log F
 \right)
 \mu ( dt ) 
 \right] 
, 
\end{eqnarray*}
 and we use the relation 
$$ 
 \bigg| \frac{D_t F }{F} \bigg|^2
 + 2 D_t D_t \log F
 = 
 2 \frac{D_tD_t F}{F} 
 - \bigg| \frac{D_t F }{F} \bigg|^2 
 , 
 \qquad 
 t\in [0,T]. 
$$ 
\end{Proof}
\noindent 
 From the above proposition it suffices that $F$ be $\Delta$-superharmonic for 
 $X + D \log F$ to be superefficient. 
 In this case we have 
\begin{equation} 
\label{sh} 
 \E_u 
 \left[ 
 \Vert X + 
 D \log F - u \Vert_{L^2([0,T] , d\mu )}^2
 \right] 
 \leq 
 \crb (\gamma ,\mu, \hat{u} ) 
 - 
 \E_u \left[
 \Vert 
 D \log F \Vert_{L^2([0,T] , d\mu )}^2 
 \right] 
, 
\end{equation} 
 with equality in \eqref{sh} when $F$ is $\Delta$-harmonic. 
\\

\noindent 
 In the next proposition we show 
 that the $\Delta$-superharmonicity of $F$ is not necessary 
 for $X + D \log F$ to be superefficient, namely 
 the $\Delta$-superharmonicity  
 of $F$ can be replaced by the $\Delta$-superharmonicity 
 of $\sqrt{F}$, which is a weaker assumption, see \cite{fourdrinier} 
 in the finite dimensional case. 
 In particular, $X+D \log F$ is a superefficient estimator of $u$ 
 if $\Delta \sqrt{F} < 0$ on a set of strictly positive $\P$-measure. 
\begin{prop} 
\label{prop:erreurcylindricalfunctions} 
 Stein-type estimator. 
 For any $\P$-a.s. positive random variable 
 $F\in \Dom (\nabla)$ such that $D_t F\in \Dom (\nabla)$, $t\in [0,T]$, 
 and 
 $(D_tD_tF )_{t\in [0,T]}\in L^1 (\Omega \times [0,T] , \P_u \otimes \mu )$, we have 
\begin{equation} 
\label{hjk} 
 \E_u 
 \left[ 
 \Vert X + 
 D \log F - u \Vert_{L^2 ([0,T] , d\mu )}^2
 \right] 
 = 
 \crb (\gamma ,\mu, \hat{u} ) 
 + 
 4 
 \E_u 
 \left[ \frac{\Delta \sqrt{F}}{\sqrt{F}} \right]
.
\end{equation} 
\end{prop} 
\begin{Proof} 
 For any $F\in \Dom ( \nabla )$ such that $F>0$, $\P$-a.s., 
 and $\sqrt{F} \in \Dom (\Delta^\vide)$, we have 
$$ 
 2 \frac{D_tD_t F}{F} - \bigg| \frac{D_t F }{F} \bigg|^2
 = 
 \frac{2}{\sqrt{F}}
 D_t \left( \frac{D_t F}{\sqrt{F}} \right) 
 = 
 \frac{4}{\sqrt{F}}
 D_t D_t \sqrt{F}, 
 \qquad t\in [0,T] 
, 
$$ 
 which implies 
\begin{equation} 
\label{qsd} 
 4 
 \frac{\Delta \sqrt{F}}{\sqrt{F}}
 = 
 2 \frac{ \Delta F }{F} 
 - \int_0^T 
 | D_t \log F |^2 
 \mu ( dt ) 
, 
\end{equation} 
 and allows us to conclude from Lemma~\ref{lemma1}. 
\end{Proof} 
\noindent 
 Relation \eqref{hjk} extends to any $F\in \Dom (\nabla )$ such that 
 $\sqrt{F} \in \Dom (\Delta )$, and $F \geq 0$, $\Delta \sqrt{F} \leq 0$, $\P$-a.s. 
\\

\noindent 
 In case $(X_t)_{t\in [0,T]}$ is a Brownian motion with constant 
 variance $\sigma_t = \sigma$, $t\in [0,T]$, we have 
\begin{equation} 
\label{sh2} 
 \E_u 
 \left[ 
 \Vert X + 
 D \log F - u \Vert_{L^2([0,T])}^2
 \right] 
 \leq 
 \frac{\sigma^2T^2}{2} 
 + 
 4 
 \E_u \left[
 \frac{\Delta \sqrt{F}}{\sqrt{F}} 
 \right] 
. 
\end{equation} 
 Given $(X_t^{\sigma,1})_{t\in [0,T]}, \ldots , (X_t^{\sigma,N})_{t\in [0,T]}$ are 
 $N$ independent samples of $(X_t)_{t\in [0,T]}$, the process 
 $\bar{X}$ defined in \eqref{xbar} satisfies 
$$ 
 \E^{\sigma/N}_u
 \left[ 
 \Vert \bar{X} + 
 D \log F - u \Vert_{L^2([0,T])}^2
 \right] 
 = 
 \frac{1}{N} 
 \crb (\sigma,\mu, \hat{u} ) 
 + 
 \frac{4}{N^2} 
 \E_u \left[ 
 \frac{\Delta \sqrt{F}}{\sqrt{F}} 
 \right] 
. 
$$ 

\noindent 
 As in \cite{stein}, the superefficient estimators constructed in this way 
 are minimax in the sense that from Proposition~\ref{minimax} 
 and Proposition~\ref{prop:erreurcylindricalfunctions.1}, 
 for all $u\in H$ 
 we have 
$$\E_u 
 \left[ 
 \Vert X + D \log F - u \Vert_{L^2([0,T], \mu )}^2
 \right] 
 < 
 \crb (\gamma ,\mu, \hat{u} ) 
  = \inf_\xi 
 \sup_{v \in \Omega } 
 \E_v \left[ \int_0^T | \xi_t - v_t|^2 dt \right] 
, 
$$ 
 provided $\Delta \sqrt{F} < 0$ on a set of strictly positive $\P$-measure, 
 thus showing that the minimax estimator 
 $\hat{u} = (X_t)_{t\in [0,T]}$ is inadmissible. 
\\ 

\noindent 

\noindent 
 Both estimators $X_t+D_t \log F$ and 
 $X_t +\E_u [D_t \log F \mid {\cal F}_t]$ 
 have bias 
$$b_t = \E_u [X_t + D_t \log F -u_t] 
 = \E_u [ D_t \log F ] 
, \qquad 
 t\in [0,T], 
$$ 
 which can be bounded as follows from  \eqref{qsd}: 
\begin{eqnarray*} 
\Vert b \Vert_{L^2([0,T], \mu )}^2 
 & = & 
 \int_0^T 
 | 
 \E_u [ D_t \log F ] 
 |^2 dt 
\\ 
 & \leq & 
 \E_u \left[ \int_0^T 
 | D_t \log F |^2 dt 
 \right] 
\\ 
 & =& 
 2 
 \E_u \left[ 
 \frac{ \Delta F }{F} 
 \right] 
 - 
 4 
 \E_u \left[ 
 \frac{\Delta \sqrt{F}}{\sqrt{F}} 
 \right] 
. 
\end{eqnarray*} 
\begin{remark} 
 In the independent increment case of Section~\ref{2}, 
 the formulas obtained in this section also hold for $u$ an adapted 
 process in $L^2 (\Omega \times [0,T] )$. 
 However, in this case the computation of the gradient $D \log F$ 
 requires in principle the knowledge of $X^u$, except when $u$ is 
 deterministic, in which case the knowledge of $X$ is sufficient. 
 Thus, assuming $u$ to be deterministic will be necessary 
 for the applications of Section~\ref{4}. 
\end{remark} 
\section{Superharmonic functionals} 
\label{app} 
 In this section we give examples of nonnegative superharmonic
 functionals with respect to the Laplacian $\Delta$. 
 We start by reviewing the construction of such functionals 
 using potential theory on the Gaussian space $(\Omega , H  , \P )$, 
 and next we turn to cylindrical functionals which will be used 
 in the numerical applications of Section~\ref{4}. 
 We assume that $( \Gamma h_k )_{k \geq 1}$ is orthogonal in $L^2 ([0,T ] , d\mu )$, 
 and we let 
$$ 
 \lambda_k = \Vert \Gamma h_k \Vert_{L^2([0,T],\mu)}, 
\qquad 
 k \geq 1. 
$$ 
 The sequence $(h_k)_{k\geq 1}$ can be realized as the solution of the eigenvalue problem 
\begin{equation} 
\label{*1} 
 \Gamma h_k = - \lambda^2_k \ddot{h}_k, 
 \qquad 
 \dot{h}_k (T) = 0, \qquad 
 k\geq 1, 
\end{equation} 
 in case $(A)$, provided $\mu (dt)=dt$, and 
$$\Gamma h_k = \lambda^2_k h_k, \qquad 
 k\geq 1, 
$$ 
 in case $(B)$ for general $\mu$. 

\subsubsection*{Potentials} 
 We refer to \cite{grossp} and \cite{Goodman} for
 the notion of harmonicity on the Wiener space with respect to
 the Gross Laplacian. 
\noindent From our orthonormality assumption on $(h_k)_{k\geq 1}$, 
 the Laplacian $\Delta$ is written as 
$$ 
 \Delta^\vide F 
 = 
 {\sum_{i=1}^n
 \partial_i^2 f_n 
 \left(
  \int_0^T \dot{h}_1 ( s ) dX^u_s 
 ,
 \ldots
 ,
 \int_0^T \dot{h}_n( s ) dX^u_s 
 \right)}
$$ 
 on cylindrical functionals. 
 Let $(W^\Omega_t)_{t \geq 0}$ denote the standard $\Omega$-valued Wiener
 process with generator $\frac{1}{2} \Delta_G$ 
 on $(\tilde{\Omega} , \tilde{\cal F} , \tilde{\P})$, represented as
\begin{equation}
\label{cv} 
 W^\Omega_t = \sum_{n=1}^\infty
 \int_0^\cdot \dot{h}_n(s) \sigma^2_s ds
 \frac{\beta_n ( t ) }{\Vert h_n \Vert_{H^\vide}}
, 
 \qquad t\in \real_+,
\end{equation} 
 where $(\beta_n(t))_{t\in \real_+}$, $n\geq 1$, 
 are independent standard Brownian motions 
 on $(\tilde{\Omega} , \tilde{\cal F} , \tilde{\P})$, 
 given as 
$$\beta_n ( t ) = \int_0^T 
 \frac{\dot{h}_n (r)}{\Vert h_n \Vert_{H^\vide}} 
 dW^\Omega_t( r )
,
 \qquad t \in \real_+,
 \quad 
 n\geq 1
.
$$ 
 We have the covariance relation 
$$\tilde{\E} 
 \left[ 
 \int_0^T 
 \dot{v}_1 (r) dW^\Omega_s (r) 
 \int_0^T 
 \dot{v}_2 (r) dW^\Omega_t (r) 
 \right] = 
 ( s \wedge t ) 
 \langle 
 v_1 , 
 v_2 
 \rangle_{H^\vide} 
,
 \qquad
 s,t \in \real_+ 
, 
 \quad
 v_1, v_2 \in H 
. 
$$
 In other terms we have 
\begin{eqnarray*} 
\lefteqn{ 
 \tilde{\E}
 [ W^\Omega_t ( a ) W^\Omega_s ( b ) ]
 = 
 ( s\wedge t )
 \sum_{n=1}^\infty
 \frac{ 
 \int_0^a \dot{h}_n(s) \sigma^2_s ds 
 \int_0^b \dot{h}_n(s) \sigma^2_s ds 
 }{\Vert h_n \Vert_{H^\vide}^2}
} 
\\ 
 & = &
 ( s\wedge t )
\left<
 \sum_{n=1}^\infty
 \frac{\dot{h}_n}{\Vert h_n \Vert_{H^\vide}^2}
 \int_0^a \dot{h}_n(s) \sigma^2_s ds 
 , 
 \sum_{n=1}^\infty 
 \frac{\dot{h}_n}{\Vert h_n \Vert_{H^\vide}^2} 
 \int_0^b \dot{h}_n(s) \sigma^2_s ds 
 \right>_{L^2([0,T],\sigma^2_t dt)} 
\\ 
 & = &
 ( s\wedge t )
 \langle {\bf 1}_{[0,a]}
 ,
 {\bf 1}_{[0,b]}
 \rangle_{L^2([0,T],\sigma^2_t dt)}
\\
 & = &
 ( s \wedge t )
 \int_0^{a \wedge b} 
 \sigma^2_r dr
, \qquad 0 \leq a, b \leq T, 
 \quad s,t \in \real_+,
\end{eqnarray*} 
 which shows that $(W^\Omega_t(a))_{a\in [0,T]}$ is a continuous Gaussian 
 martingale with quadratic variation $\sigma^2_a da$ 
 for fixed $t\in \real_+$. 
\\ 

\noindent
 Denote by $(B_t)_{t\in \real_+}$ the $H^\vide $-valued
 Wiener process represented as
$$ 
 B_t
 =
 \sum_{n=0}^\infty
 \frac{h_n}{\Vert h_n \Vert_{H^\vide}^2}
 \beta_n ( t )
,
$$ 
 with $\beta_n ( t ) = \langle B_t , h_n \rangle_H$, 
 $n\geq 1$, and covariance
$$\tilde{\E} [\langle B_s , h_n \rangle_{H^\vide } 
 \langle B_t , h_m \rangle_{H^\vide } ] =
 {\bf 1}_{\{ n = m \}} ( s \wedge t ),
 \qquad
 s,t \in \real_+,
$$
 i.e.
$$
 \tilde{\E}
 [\langle B_t , v_1 \rangle_{H^\vide }
 \langle B_s ,v_2 \rangle_{H^\vide } ]
 = ( s \wedge t )
 \langle Q v_1, v_2\rangle_{H^\vide}
,
 \qquad
 s,t \in \real_+,
 \quad
 v_1,v_2\in H^\vide ,
$$
 where $Q:H^\vide  \to H^\vide $ is the operator with eigenvalues
 $\{ \Vert h_n \Vert_{H^\vide}^{-2} \ : \ n \geq 1 \}$
 in the Hilbert basis $(h_n)_{n\geq 1}$.
 It\^o's formula for Hilbert-valued Wiener processes, cf.
 Theorem~4.17 of \cite{daprato}, shows that
$$F(B_t) 
 = F(B_0)
 +
 \int_0^t
 \langle D F (B_s) , dB_s \rangle_{H^\vide}
 +
 \frac{1}{2}
 \int_0^t
 \Delta^\vide F (B_s)
 ds
, \qquad
 F\in {\cal S},
$$
 hence $(B_t)_{t\in \real_+}$ has generator $\frac{1}{2} \Delta^\vide$.
\\

\noindent
 Dynkin's formula, cf. \cite{Dynkin}, Theorem~5.1, shows that
 for all stopping time $\tau$ such that
 $\tilde{\E} [\tau \mid B_0 = \omega ] < \infty$ we have,
 $\P^\sigma (d\omega )$-a.s.: 
$$
 \tilde{\E}
 [
 F ( B_\tau )
 \mid B_0 = \omega
 ]
 -
 F ( \omega )
 =
 \frac{1}{2}
 \tilde{\E}
 \left[
 \int_0^\tau
 \Delta^\vide F ( B_s )
 ds
 \mid B_0 = \omega
 \right]
,
$$
 hence $\Delta^\vide F \leq 0$ implies
$$
 F(\omega )
 \geq
 \tilde{\E} [
 F( B_\tau ) \mid B_0 = \omega
 ]
.
$$
 For $r>0$, let
$$\tau_{r}
 = \inf \{ t \in \real_+ \ : \ B_t \notin {\eufrak B}_r (B_0)\}
$$
 denotes the first exit time of $(B_t)_{t\in [0,T]}$ from the open ball
 ${\eufrak B}_r (\omega )$ of radius $r>0$, centered at $B_0 = \omega \in \Omega$.
 We have the following converse.
\begin{prop} 
 Let $F\in \Dom (\Delta^\vide )$ be such that $\Delta^\vide F$ is continuous 
 on $\Omega$, and assume that there exists $r_0>0$ such that
\begin{equation} 
\label{jkl} 
 F(\omega )
 \geq
  \tilde{\E} [
 F( B_{\tau_r} ) \mid B_0 = \omega
 ]
, \qquad
 \P^\sigma_u (d\omega )-a.s., 
 \quad 
 0 < r < r_0
. 
\end{equation} 
 Then $F$ is $\Delta^\vide$-superharmonic on $\Omega$ in the sense of 
 Relation \eqref{sph}. 
\end{prop}
\begin{Proof}
 From Remark~3, page 134 of \cite{Dynkin}, we have
\begin{equation}
\label{dk}
 \frac{1}{2}
 \Delta^\vide F ( \omega )
 =
 \lim_{n\to \infty}
 \frac{ \tilde{\E}
 [
 F ( B_{\tau_{1/n}} )
 \mid B_0 = \omega
 ]
 - 
 F ( \omega )
}{\tilde{\E} [\tau_{1/n} \mid B_0 = \omega ]} ,
\end{equation}
 which shows that $\Delta^\vide F \leq 0$ when \eqref{jkl}
 is satisfied.
\end{Proof}
\noindent 
 This yields in particular the following class of 
 $\Delta^\vide$-superharmonic functionals. 
\begin{prop} 
 Let the potential of $F\geq 0$ be defined by 
\begin{equation} 
\label{hj} 
 G (\omega ) = \int_0^{+\infty}
 \tilde{\E}
 [
 F ( B_t ) \mid B_0 = \omega
 ]
 dt 
, \qquad
 \P^\sigma_u (d\omega )-a.s. 
, 
\end{equation} 
 assume that $G\in \Dom ( \Delta^\vide ) $ and 
 that $\Delta^\vide G$ is continuous on $\Omega$. 
 Then $G$ is a $\Delta^\vide$-superharmonic on $\Omega$. 
\end{prop}
\begin{Proof}
 For all $r>0$ we have
\begin{eqnarray*} 
 G (\omega )
 & = & 
 \tilde{\E}
 \left[
 \int_0^{\tau_r} F ( B_t )  dt
 \Big|
 B_0 = \omega
 \right]
 +
 \tilde{\E}
 [
 G ( B_{\tau_r } )
 \mid
 B_0 = \omega
 ]
\\ 
& \geq & 
 \tilde{\E}
 [
 G ( B_{\tau_r } )
 \mid
 B_0 = \omega
 ]
, 
\end{eqnarray*} 
 which shows that $G$ is $\Delta^\vide$-superharmonic.
\end{Proof}
\noindent 
 Note that if $F$ is bounded with bounded support 
 in $\Omega$ then $G$ is bounded on $\Omega$, 
 see e.g. Remark~3.5 of \cite{grossp}. 
\\ 

\subsubsection*{Convolution} 
\noindent 
 Positive superharmonic functionals can also be obtained 
 by convolution, i.e. if 
 $F$ is $\Delta^\vide$-superharmonic and $G$ is positive and 
 sufficiently integrable, then
$$
 \omega \mapsto
 \int_\Omega
 G ( \tilde{\omega} )
 F ( \omega - \tilde{\omega} )
 \P^\sigma ( d \tilde{\omega} )
$$ 
 is positive and $\Delta^\vide$-superharmonic.
\subsubsection*{Cylindrical functionals} 
\noindent 
 Superharmonic functionals on Gaussian space can also
 be constructed as cylindrical functionals, 
 by composition with finite-dimensional functions. 
 Here we use the expansions of case $(A)$. 
 From the expression of $\Delta$ on cylindrical functionals 
$$ 
 \Delta F 
 = 
 \sum_{i=1}^n
 \partial_i^2 f_n 
 \left( 
 \lambda_1^{-1} 
 X^u ( h_1 ) 
 ,
 \ldots
 , 
 \lambda_n^{-1} 
 X^u ( h_n ) 
 \right) 
, 
$$ 
 we check that 
$$F 
 = 
 f_n 
 \left( 
 \lambda_1^{-1} 
 X^u ( h_1 ) 
 ,
 \ldots
 , 
 \lambda_n^{-1} 
 X^u ( h_n ) 
 \right) 
$$ 
 is superharmonic on $\Omega$ if and only if $f_n$ is 
 superharmonic on $\real^n$. 
 Given $a\in \real$ and $b\in \real^n$, let
 $f_{n,a,b} : \real^n \to \real$ be defined as
$$ 
 f_{n,a,b} (x_1,\ldots ,x_n) = 
 \Vert x + b \Vert^a 
 = 
 ((x_1+b_1)^2+\cdots + (x_n+b_n)^2)^{a/2}
,
$$ 
 then $\sqrt{f_{n,a,b}}$ is superharmonic on $\real^n$, 
 $n\geq 3$, if and only if $a\in [4-2n,0]$. 
 Let 
$$F_{n,a,b} = f_{n,a,b} \left( 
 \lambda_1^{-1} X^u ( h_1 ) 
 ,
 \ldots
 , 
 \lambda_n^{-1} X^u ( h_n ) 
 \right) 
. 
$$ 
 We have 
$$ 
 D_t \log F_{n,a,b} 
 = 
 a 
 \sum_{i=1}^n
 \frac{
 \lambda_i^{-1} 
 \Gamma h_i (t) 
 \left( 
 b_i 
 + 
 \lambda_i^{-1} X^u ( h_i ) 
 \right) 
}{
 \left|
  b_1 + 
 \lambda_1^{-1} 
 X^u ( h_1 ) 
 \right|^2
 +
 \cdots
 +
 \left|
 b_n 
 + 
 \lambda_n^{-1} 
 X^u ( h_n ) 
 \right|^2
} 
, 
$$ 
 and 
\begin{eqnarray*} 
 \Delta \sqrt{F_{n,a,b}}
 & = & 
 \sum_{i=1}^n
 \partial_i^2
 \sqrt{f_{n,a,b}} 
 \left(
 \lambda_1^{-1} X^u ( h_1 ) 
 , 
 \ldots
 , 
 \lambda_n^{-1} X^u ( h_n ) 
 \right)
, 
\end{eqnarray*} 
 since $( \Gamma h_k )_{k \geq 1}$ is orthogonal in 
 $L^2([0,T] , dt )$, hence 
\begin{eqnarray*} 
\label{eq:laplacienracineFsurracineF} 
 \frac{\Delta 
 \sqrt{F_{n,a,b}}}{\sqrt{F_{n,a,b}}} 
 & = & 
 \frac{ 
 a ( n -2 + a/2 )/2}{ 
 \left| 
 b_1 
 + 
 \lambda_1^{-1} X^u ( h_1 ) 
 \right|^2 
 + 
 \cdots 
 + 
 \left| 
 b_n 
 + 
 \lambda_n^{-1} 
 X^u ( h_n ) 
 \right|^2 
} 
, 
\end{eqnarray*} 
 is negative 
 if $4-2n \leq a \leq 0$, which is minimal for $a=2-n$. 
 We also have 
\begin{eqnarray*} 
 \frac{\Delta 
 F_{n,a,b}}{F_{n,a,b}} 
 & = & 
 \frac{ 
 a ( n+ a - 2 )}{
 \left| 
 b_1 
 + 
 \lambda_1^{-1} 
 X^u ( h_1 ) 
 \right|^2 
 + 
 \cdots 
 + 
 \left| 
 b_n 
 + 
 \lambda_n^{-1} 
 X^u ( h_n ) 
 \right|^2 
} 
, 
\end{eqnarray*} 
 which is negative for $a\in [2-n,0]$ and vanishes for $a=2-n$. 
 In this case the estimator is given by 
$$ 
 D_t \log F_{n,2-n,b} 
 = 
 - 
 (n-2) 
 \sum_{i=1}^n 
 \frac{ 
 \lambda_i^{-1} 
 \left( 
 b_i 
 + 
 \lambda_i^{-1} 
 X^u ( h_i ) 
 \right) 
 \Gamma h_i ( t ) 
}{ 
 \left| 
 b_1 + 
 \lambda_1^{-1} 
 X^u ( h_1 ) 
 \right|^2 
 + 
 \cdots 
 + 
 \left| 
 b_n 
 + 
 \lambda_n^{-1} 
 X^u ( h_n ) 
 \right|^2 
} 
, 
$$
 and from Proposition~\ref{prop:erreurcylindricalfunctions.1}, 
 inequality \eqref{sh} actually also holds as an equality: 
\begin{equation} 
\label{hldseq} 
 \E_u 
 \left[ 
 \Vert X + 
 D \log F_{n,2-n,b} - u \Vert_{L^2([0,T] , dt )}^2 
 \right] 
 = 
 \crb (\sigma,\mu, \hat{u} ) 
 - 
 \E_u \left[ 
 \int_0^T 
 | D_t \log F_{n,2-n,b} |^2 
 dt \right] 
, 
\end{equation} 
 with 
\begin{equation} 
\label{eq:est2} 
 \Vert D \log F_{n,2-n,b} \Vert_{L^2([0,T] , dt )}^2 
 = 
 \frac{ (n-2)^2 }{ \left| 
 b_1 + 
 \lambda_1^{-1} 
 X^u ( h_1 ) 
 \right|^2 
 + 
 \cdots 
 + 
 \left| 
 b_n 
 + 
 \lambda_n^{-1} 
 X^u ( h_n ) 
 \right|^2 
} 
. 
\end{equation} 
\noindent 
 Note that when $u$ is deterministic, any superharmonic functional of the form 
$$ 
 f_n 
 \left(
 \lambda_1^{-1} 
 X^u ( h_1 ) 
 ,
 \ldots
 ,
 \lambda_n^{-1} 
 X^u ( h_n ) 
 \right) 
, 
$$ 
 can be replaced with 
$$ 
 f_n 
 \left(
 \lambda_1^{-1} 
 X ( h_1 ) 
 ,
 \ldots
 ,
 \lambda_n^{-1} 
 X ( h_n ) 
 \right) 
, 
$$ 
 which retains the same harmonicity property, 
 and can be directly computed from an observation of $X$. 
\\ 

\noindent 
 The Stein type estimator of $u$ is given by 
$$ 
 X_t + D_t \log F_{n,2-n,b}, \qquad t \in [0,T], 
$$ 
 with 
$$ 
 b_i = \lambda_i^{-1} \langle u , h_i\rangle, 
 \qquad i=1,\ldots , n, 
$$ 
 i.e. 
$$ 
 D_t \log F_{n,2-n,b} 
 = 
 - 
 (n-2) 
 \frac{[\Pi_n X]_t }{\Vert \Pi_n X \Vert_{L^2([0,T] , dt )}^2} 
, 
$$ 
 where $\Pi_n$ denotes the orthogonal projection 
$$ 
 \Pi_n X (t) 
 : = 
 \sum_{k=1}^n 
 \lambda_k^{-1} 
 X (h_k) 
 \Gamma h_k (t) 
 = 
 \sum_{k=1}^n 
 \lambda_k^{-1} 
 \left( 
 b_k 
 +  
 \lambda_k^{-1} 
 X^u( h_k ) 
 \right) 
 \Gamma h_k (t) 
. 
$$ 
 We have 
\begin{eqnarray*} 
 \Vert D \log F_{n,2-n,b} \Vert_{L^2 ([0,T]\times \Omega , \P_u \otimes dt )}^2 
 & = & 
 - 4 
 \E_u \left[ 
 \frac{\Delta \sqrt{F_{n,2-n,b}}}{\sqrt{F_{n,2-n,b}}} \right]
\\ 
 & = & 
 (n-2)^2 
 \E_u \left[ 
 \frac{ 1 }{ \left|
 \lambda_1^{-1} X ( h_1 ) 
 \right|^2 
 + 
 \cdots
 +
 \left| 
 \lambda_n^{-1}  X ( h_n ) 
 \right|^2
} 
 \right] 
\\ 
 & = & 
 (n-2)^2 
 \E_u \left[ 
 \Vert \Pi_n X \Vert_{L^2([0,T] , dt )}^{-2} 
 \right] 
, 
\end{eqnarray*} 
 and 
$$ 
 \E_u
 \left[
 \Vert X 
 + 
 D \log F_{n,2-n,b} - u \Vert_{L^2([0,T] , dt )}^2
 \right] 
 = 
 \crb (\sigma,\mu, \hat{u} ) 
 - 
 (n-2)^2 
 \E_u \left[ 
 \Vert \Pi_n X \Vert_{L^2([0,T] , dt )}^{-2} 
 \right] 
. 
$$ 
 Note that the estimator 
$$ 
 X_t - (n-2) 
 \frac{[\Pi_n X]_t }{\Vert \Pi_n X \Vert_{L^2([0,T] , dt )}^2}, 
 \qquad 
 t\in [0,T], 
$$ 
 is of James-Stein type, but it is not a shrinkage operator. 
 Another difference with James-Stein estimators is that here 
 the denominator consists in a sum of squared Gaussians with different 
 variances. 
\\ 

\noindent 
 Given $(X_t^1)_{t\in [0,T]}, \ldots , (X_t^N )_{t\in [0,T]}$, 
 $N$ independent samples of $(X_t)_{t\in [0,T]}$, the process 
$$\bar{X}_t = \frac{1}{N} 
 \left( 
 X_t^1 + \cdots + X_t^N \right) 
$$ 
 is a Brownian motion with drift $u$ and 
 quadratic variation 
 $\sigma^2_t dt /N$ under $\P_u$, and 
 can be used for both efficient and Stein type 
 estimation. 
\section{Numerical application} 
\label{4} 
\noindent 
 In this section we present numerical simulations which 
 allow us to measure the efficiency of our estimators. 
 We use the framework of case $(A)$ and the 
 superharmonic functionals constructed 
 as cylindrical functionals in the previous section, 
 and we assume that $u \in H$ is deterministic. 
\\ 

\noindent 
 We work in the independent increment framework of Section~\ref{2} 
 and we additionally assume that $\sigma_t = \sigma$ is constant, $t\in [0,T]$, 
 i.e. $(X_t)_{t\in [0,T]}$ is a Brownian motion with variance $\sigma^2$, 
 $\Gamma h (t)= \sigma^2h(t)$, $t\in [0,T]$, and 
$$\crb (\sigma,\mu, \hat{u} ) = \frac{\sigma^2}{2} T^2. 
$$ 
 Letting 
$$ 
 h_n (t) =
\frac{\sqrt{2T}}{
 \sigma 
 \pi 
 \left( n- 1/ 2 \right) 
} 
 \sin \left( \left( 
 n-\frac{1}{2} \right) \frac{ \pi t}{T} 
 \right), 
 \qquad 
 t \in [0,T]
, 
 \quad
  n \geq 1, 
$$ 
 i.e. 
$$\dot{h}_n (t) = 
 \frac{1}{\sigma} 
 \sqrt{\frac{2}{T}} 
 \cos \left( 
 \left( 
 n-\frac{1}{2} \right) \frac{ \pi t}{T} 
 \right), 
 \qquad 
 t \in [0,T], 
 \quad n \geq 1, 
$$ 
 provides an orthonormal basis $(h_n)_{n\geq 1}$ of $H$ 
 such that $(\Gamma h_k)_{k\geq 1}$ is orthogonal in $L^2([0,T] , dt )$, with 
$$ 
 \lambda_n 
 = 
 \frac{\sigma T}{ 
 \pi 
 ( n-1/2 ) }, \qquad n\geq 1 
, 
$$ 
 solution of \eqref{*1}. 
 The estimator of $u$ will be given by 
$$ 
 D_t \log F_{n,2-n,b} 
 = 
 -
 (n-2) 
 \sqrt{\frac{2}{T}} 
 \sum_{k=1}^n
 \frac{ 
 \displaystyle 
 X ( 
 h_k 
 ) 
 }{
 \left|
 \lambda_1^{-1} 
 X ( h_1 ) 
 \right|^2 
 +
 \cdots
 +
 \left| 
 \lambda_n^{-1} X ( h_n ) 
 \right|^2 
} 
 \sin \left( 
 \left( 
 k-\frac{1}{2} 
 \right) 
 \frac{ \pi t}{T} \right) 
, 
$$ 
 For simulation purposes we will use $X+D\log F$, and 
 construct the (nondrifted) Brownian motion $(X^u_t)_{t\in [0,T]}$ 
 via the Paley-Wiener expansion 
\begin{equation} 
\label{pw} 
 X^u_t 
 = \sigma^2 
 \sum_{n=1}^\infty \eta_n h_n (t) 
 = 
 \sigma 
 \frac{\sqrt{2T}}{\pi} 
 \sum_{n=1}^\infty \eta_n 
 \frac{ 
 \sin \left( 
 \left( 
 n-\frac{1}{2} \right) \frac{ \pi t}{T} 
 \right)}{ 
 \left( 
 n-\frac{1}{2} 
 \right) 
 } 
, 
\end{equation} 
 where $(\eta_n)_{n\geq 1}$ are independent standard Gaussian 
 random variables 
 with unit variance under $\P_u$ and 
$$ 
 \eta_n 
 = 
 \int_0^T \dot{h}_n (s) dX^u_s 
, 
 \qquad 
 n\geq 1 
. 
$$ 
 In this case we have 
\begin{eqnarray} 
\label{hldseq1} 
 \lefteqn{ 
 D_t \log F_{n,2-n,b} 
} 
\\ 
\nonumber 
 & = & 
 -
 (n-2) 
 \sqrt{\frac{2}{T}} 
 \sum_{k=1}^n
 \frac{ 
 \displaystyle 
 \eta_k 
 + 
 \langle u , h_k \rangle 
 }{ 
 \displaystyle 
 \sum_{l=1}^n 
 \lambda_l^{-2} 
 \left( 
 \eta_l 
 + 
 \langle u , h_l \rangle 
 \right)^2 
} 
 \sin \left( 
 \left(k-\frac{1}{2} \right) \frac{ \pi t}{T} 
 \right) 
. 
\end{eqnarray} 
\noindent 
\noindent 
 Recall that the improvement obtained in comparison with the 
 efficient estimator $\hat{u}$ is not obtained
 pathwise, but in expectation.
 The gain of the superefficient estimator 
 $X + D \log F_{n,2-n,b}$ compared to the 
 efficient estimator $\hat{u}$ is given by 
$$ 
 G ( u , \sigma , T , n ) 
 : 
 = 
 - \frac{4}{ \crb (\sigma,\mu, \hat{u} ) } 
 \E_u \left[ \frac{\Delta \sqrt{F_{n,2-n,b}}}{\sqrt{F_{n,2-n,b}}} \right] 
$$ 
 as a function of $n \geq 3$. 
 From \eqref{hldseq} and \eqref{hldseq1} we have 
\begin{equation} 
\label{jkl2} 
 G ( u , \sigma , T , n ) 
 = 
 2 
 (n-2)^2 
 \E \left[ 
 \left( 
 \displaystyle 
 \sum_{l=1}^n 
 \left( 
 \pi 
 \left( 
 l 
 - 
 \frac{1}{2} 
 \right) 
 \left( 
 \eta_l 
 + 
 \langle u , h_l \rangle 
 \right) 
 \right)^2 
 \right)^{-1} 
 \right] 
, 
\end{equation} 
 hence $G(u,\sigma , T , n)$ converges to 
\begin{equation} 
\label{gsint} 
 (n-2)^2 
 \frac{8}{\pi^2} 
 \E \left[ 
 \left( 
 \displaystyle 
 \sum_{l=1}^n 
 \left( 
 2 l 
 - 
 1 
 \right)^2 
 \eta_l^2 
 \right)^{-1} 
 \right] 
, 
\end{equation} 
 as $\sigma$ tends to infinity. 
 The quantity \eqref{gsint} can be evaluated as a Gaussian integral 
 to yield \eqref{gsint1}. 
 Unlike in the classical Stein method, we stress that here $n$ becomes 
 a free parameter and there is some interest in determining the values 
 of $n$ which yield the best performance. 
\begin{prop} 
\label{prop:limitesestimateur} 
 For all $\sigma , T>0$, and $u\in H$ we have 
$$ 
 G( u , \sigma , T, n ) 
 \simeq 
 \frac{6}{ n \pi^2} 
$$ 
 as $n$ goes to infinity. 
\end{prop} 
\begin{Proof} 
 Let 
$$S_n = 
 \sum_{l=1}^n 
 \left( 
 \pi 
 \left( 
 n - l 
 + 
 \frac{1}{2} 
 \right) 
 \left( 
 \eta_l 
 + 
 \langle u , h_l \rangle 
 \right) 
 \right)^2 
, 
 \qquad 
 n \geq 1 
. 
$$ 
 We have 
$$ 
 G ( \alpha , \sigma , T , n ) 
 = 
 2 (n-2)^2 
 \E \left[ 
 \frac{1}{ 
 S_n} 
 \right] 
, 
$$ 
 and by the strong law of large numbers, $n (n-2)^2{S_n}^{-1}$ converges to 
 $3/\pi^2$ as $n$ goes to infinity, since 
$$ 
 \lim_{n\to \infty}
 \frac{\E [S_n]}{n^3}
 = 
 \frac{\pi^2}{4} 
 \lim_{n\to \infty}
 \frac{1}{n^3} 
 \sum_{i=1}^n (2i-1)^2 
 = 
 \frac{\pi^2}{3} 
.
$$ 
 Now for all $n >10$ we have 
\begin{eqnarray*} 
\E_u 
\left[ 
 \left( 
 \frac{(n-2)^3}{S_n} 
 \right)^2 
\right] 
 & = & 
\E 
\left[ 
\Lambda (u) 
 \left( 
 \frac{(n-2)^3}{S_n} 
 \right)^2 
\right] 
\\ 
 & \leq & 
 n^2 \pi^2 
\E 
\left[ 
\Lambda (u)^2 
\right]^{1/2} 
\E 
\left[ 
 \left( 
 \sum_{l=1}^{[n/2]} 
 \left( 
 1 - \frac{l}{n} 
 + 
 \frac{1}{2n} 
 \right)^2 
 \eta_l^2 
 \right)^{-4} 
\right]^{1/2} 
\\ 
 & \leq & 
 n^2 
 \frac{4}{\pi^4} 
\E 
\left[ 
\Lambda (u)^2 
\right]^{1/2} 
\E 
\left[ 
 \left( 
 \sum_{l=1}^{[n/2]} 
 \eta_l^2 
 \right)^{-4} 
\right]^{1/2} 
\\ 
 & \leq & 
 \frac{4n^2}{\pi^4} 
\E 
\left[ 
\Lambda (u)^2 
\right]^{1/2} 
 \left( 
 \prod_{k=1}^4 
 \left( 
 [n/2] - 2 k 
 \right) 
 \right)^{-1/4} 
, 
\end{eqnarray*} 
 hence $n^3/S_n$ is uniformly integrable in $n > 16$, 
 where $[n/2]$ denotes the integer part of $n/2$. 
 This concludes the proof. 
\end{Proof} 
\noindent 
 In the sequel we choose $u_t=\alpha t$, $t\in [0,T]$, $\alpha \in  \real$. 
 Figure~\ref{g1} gives a sample path representation of the process $X+D\log F$. 
\vskip0cm 
\begin{figure}[!ht]
\centering 
\caption{\small $u(t)=t, \; t \in [0,T]; \; n=5$.} 
\label{g1} 
\resizebox*{13cm}{7cm}{\rotatebox{0}{\includegraphics[scale=0.9]{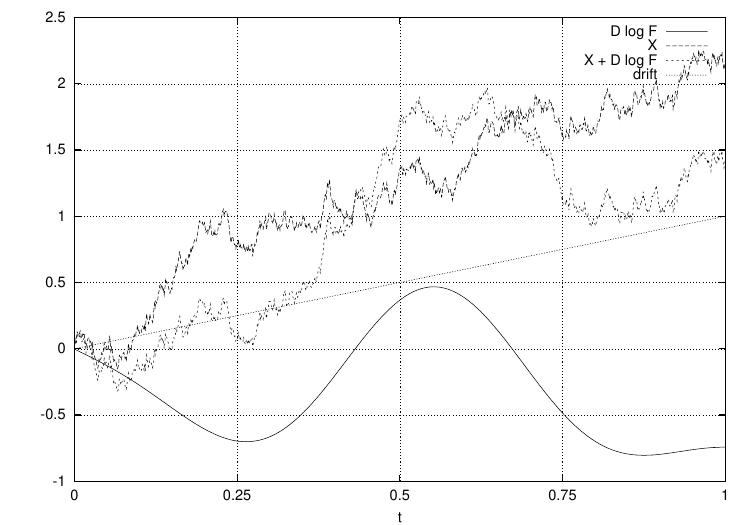}}} 
\end{figure}
\noindent 
 In this case, from \eqref{jkl2} we have 
$$ 
 G ( \alpha , \sigma , T , n ) 
 = 2 
 (n-2)^2 
 \E \left[ 
 \left( 
 \displaystyle 
 \sum_{l=1}^n 
 \left( 
 \pi 
 \left( 
 l 
 - 
 \frac{1}{2} 
 \right) 
 \eta_l 
 - 
 \alpha 
 \frac{\sqrt{2T}}{\sigma} 
 (-1)^l 
 \right)^2 
 \right)^{-1} 
 \right] 
, 
$$ 
 from which it follows that $G(\alpha , \sigma , T , n)$ converges to 
$$ 
 (n-2)^2 
 \frac{8}{\pi^2} 
 \E \left[ 
 \left( 
 \displaystyle 
 \sum_{l=1}^n 
 \left( 
 2 l 
 - 
 1 
 \right)^2 
 \eta_l^2 
 \right)^{-1} 
 \right] 
, 
$$ 
 when $\alpha^{-2}\sigma^2/T$ tends to infinity, and is equivalent to 
$$ 
 \left( 1 -\frac{2}{n} \right)^2 
 \frac{\sigma^2}{\alpha^2 T} 
$$ 
 as $\alpha^{-2}\sigma^2/T$ tends to $0$. 
\noindent 
 Figure~\ref{g3} represents the gain in percentage of the
 superefficient estimator 
 $X + \sigma^2 D \log F_{n,2-n,b}$ 
 compared to the efficient estimator $\hat{u}$ using Monte-Carlo simulations, i.e. 
 we represent $100\times G(\alpha , \sigma , T , n)$ as a function of $n\geq 3$. 
\vskip0cm 
\begin{figure}[!ht]
\centering 
\caption{\small Percentage gain as a function of $n$ for 10000 samples and $\alpha = \sigma = T=1$.} 
\label{g3} 
\resizebox*{13cm}{7cm}{\rotatebox{0}{\includegraphics[scale=0.9]{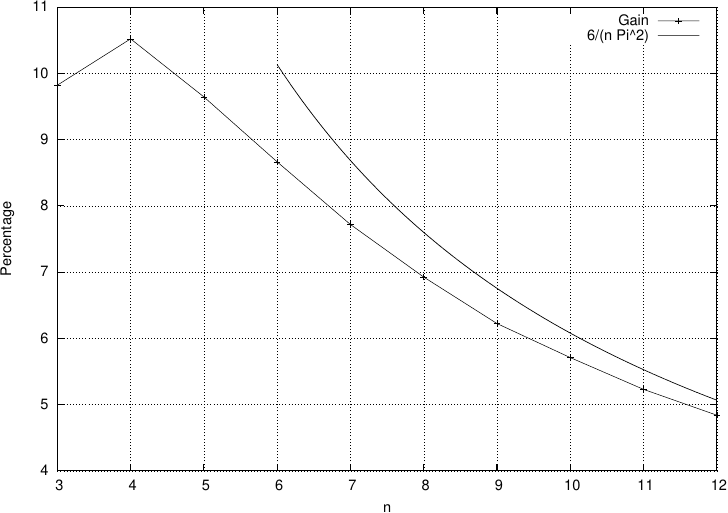}}} 
\end{figure}
\noindent 
 An optimal value 
$$ 
 n_{\rm opt} 
 = \argmax 
 \left\{ 
 G ( \alpha , \sigma , T , n ) 
 ~ : ~
 n\geq 3 
 \right\} 
$$ 
 of $n$ exists in general and is equal to $4$ when $\alpha = \sigma = T = 1$. 
\noindent 
 Figure~\ref{g4} shows the variation of the gain as a function 
 of $n$ and $T$ for $\alpha = \sigma = 1$. 
\begin{figure}[!ht]
\centering
\caption{\small Gain as a function of $n$ and $T$.} 
\label{g4} 
\resizebox*{13cm}{7cm}{\rotatebox{0}{\includegraphics[scale=0.9]{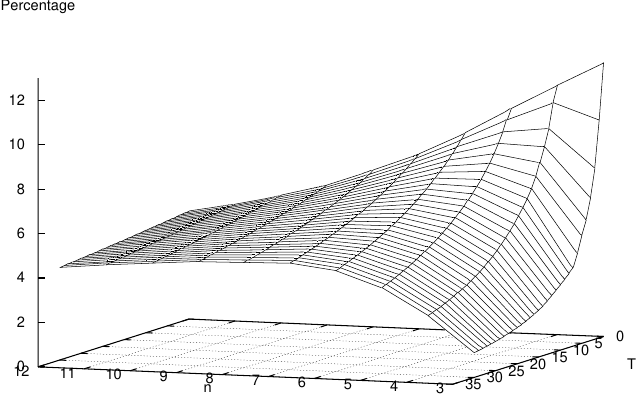}}} 
\end{figure} 
\noindent 
 Figure~\ref{g5} represents the variation of the gain as a function 
 of $n$ and $\sigma$. 
\vskip0cm 
\begin{figure}
\centering
\caption{\small Gain as a function of $n$ and $\sigma$.} 
\label{g5} 
\resizebox*{13cm}{7cm}{\rotatebox{0}{\includegraphics[scale=0.9]{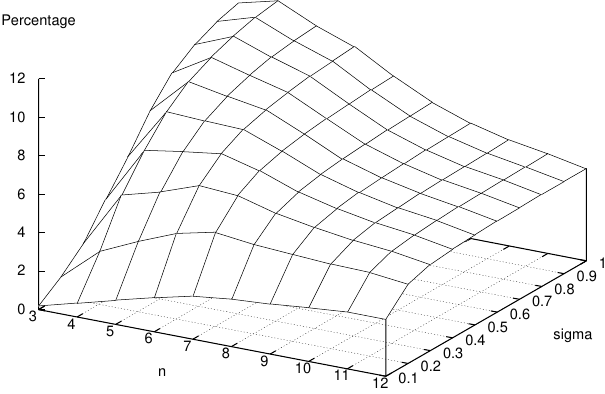}}} 
\\ 
\end{figure} 
\section{Appendix} 
 The next Proposition is classical in the framework of Gaussian filtering 
 and is needed in Section~\ref{2} for Bayes estimation. 
 Its proof is stated for completeness since we did not find it in the literature. 
\begin{prop} 
 Let $Z$ be a Gaussian process with covariance operator $\Gamma_\tau$ 
 and drift $v\in H$, 
 and assume that $X$ is a Gaussian process with drift $Z$ 
 and quadratic covariance operator $\Gamma$ given $Z$. 
 Then, conditionally to $X$, $Z$ has drift 
$$ 
 f
 \mapsto 
 \langle 
 \stackrel{}{\chi}_t 
 , 
 ( \Gamma+\Gamma_\tau)^{-1}\Gamma 
 v 
 \rangle  
 + 
 X ( 
 ( \Gamma+\Gamma_\tau)^{-1}\Gamma_\tau f 
 \stackrel{}{\chi}_t 
 ) 
 \quad 
 \mbox{and covariance} 
 \quad 
 \Gamma_\tau ( \Gamma + \Gamma_\tau )^{-1} \Gamma 
. 
$$ 
\end{prop} 
\begin{Proof} 
 For convenience of notation, let 
$$ 
 V(f) 
 := 
 \langle 
 f
 , 
 v 
 \rangle  
, 
 \qquad 
 f\in H 
. 
$$ 
 For all $f, g \in H$ we have: 
\begin{eqnarray*} 
 \E\left[ 
 \exp 
 \left( 
 i 
 X ( 
 f 
 ) 
 \right) 
 \right] 
 & = & 
 \E\left[ 
 \E\left[ 
 \exp 
 \left( 
 i 
 X ( f ) 
 \right) 
 \Big | 
 Z 
 \right] 
 \right] 
\\ 
 & = & 
 \E\left[ 
 \exp 
 \left( 
 i 
 Z ( f ) 
 - 
 \frac{1}{2} 
 \langle 
 f 
 , 
 \Gamma 
 f \rangle 
 \right) 
\right] 
\\ 
 & = & 
 \exp 
 \left( 
 i 
 V ( f ) 
 - 
 \frac{1}{2} 
 \langle 
 f 
 , 
 ( 
 \Gamma_\tau 
 + 
 \Gamma 
 ) 
 f \rangle 
\right)
, 
\end{eqnarray*} 
 and 
\begin{eqnarray*} 
\lefteqn{ 
 \E\left[ 
 \exp 
 \left( 
 i 
 X ( 
 g 
 ) 
 \right) 
 \exp 
 \left( 
 i 
 Z ( f  ) 
 \right) 
 \right] 
 } 
\\ 
 & = & 
 \E\left[ 
 \exp 
 \left( 
 i 
 Z ( f ) 
 \right) 
 \E\left[ 
 \exp 
 \left( 
 i 
 X ( g ) 
 \right) 
 \Big | 
 Z 
 \right] 
 \right] 
\\ 
 & = & 
 \E\left[ 
 \exp 
 \left( 
 i 
 Z ( 
 f + 
 g
 ) 
 - 
 \frac{1}{2} 
 \langle 
 g , \Gamma g 
 \rangle 
 \right) 
\right] 
\\ 
 & = & 
 \exp 
 \left( 
 - 
 \frac{1}{2} 
 \langle 
 f+g,
 \Gamma_\tau (f+g) 
 \rangle 
 - 
 \frac{1}{2} 
 \langle 
 g 
 , 
 \Gamma 
 g 
 \rangle 
 + 
 i 
 V ( f+g ) 
\right)
\\ 
 & = & 
 \exp 
 \left( 
 i 
 V ( 
 (\Gamma + \Gamma_\tau )^{-1} \Gamma_\tau 
 f 
 ) 
 + 
 i 
 V ( 
 g 
 + 
 (\Gamma + \Gamma_\tau )^{-1} 
 \Gamma 
 f 
 ) 
 - 
 \frac{1}{2} 
 \langle 
 \Gamma_\tau f , (\Gamma +\Gamma_\tau )^{-1} 
 \Gamma f 
 \rangle 
\right. 
\\ 
& & \left. 
 - 
 \frac{1}{2} 
 \langle 
 g 
 + 
 (\Gamma + \Gamma_\tau )^{-1} 
 \Gamma_\tau 
 f 
 , 
 (\Gamma + \Gamma_\tau) 
 ( 
 g 
 + 
 (\Gamma + \Gamma_\tau )^{-1} 
 \Gamma_\tau 
 f 
 ) 
 \rangle 
 \right)
\\  
 & = & 
 \E\left[ 
 \exp 
 \left( 
 i 
 X ( g ) 
 \right) 
\right. 
\\ 
& & 
\left. 
 \exp 
 \left( 
 i 
 X ( 
 (\Gamma + \Gamma_\tau )^{-1} 
 \Gamma_\tau f ) 
 + 
 i 
 V((\Gamma +\Gamma_\tau)^{-1}\Gamma f) 
 - 
 \frac{1}{2} 
 \langle 
 \Gamma_\tau f 
 , 
 (\Gamma_\tau + \Gamma )^{-1} 
 \Gamma 
 f 
 \rangle 
 \right)
 \right] 
, 
\end{eqnarray*} 
 which shows that 
\begin{eqnarray*} 
\lefteqn{ 
 \E\left[ 
 \exp 
 \left( 
 i 
 Z (f) 
 \right) 
 \Big | 
 X 
 \right] 
} 
\\ 
 & = & 
 \exp 
 \left( 
 i 
 V ( 
 (\Gamma + \Gamma_\tau )^{-1} 
 \Gamma 
 f ) 
 + 
 i 
 X ( 
 (\Gamma +\Gamma_\tau)^{-1}\Gamma_\tau 
 f ) 
 - 
 \frac{1}{2} 
 \langle 
 \Gamma_\tau f 
 , 
 (\Gamma + \Gamma_\tau )^{-1} 
 \Gamma 
 f \rangle 
 \right)
. 
\end{eqnarray*} 
\end{Proof} 
\noindent 
 In particular we get the following corollary 
 which is classical in the framework of Gaussian filtering. 
\begin{prop} 
\label{ljk2} 
 Let $(Z_t)_{t\in [0,T]}$ be a Brownian motion with quadratic variation $\tau^2_tdt$, 
 $\tau \in L^2 ([0,T] , dt )$, and drift $(v_t)_{t\in [0,T]}$, $v\in H$, 
 and let $(X_t)_{t\in [0,T]}$ have drift $(Z_t)_{t\in [0,T]}$ 
 and quadratic variation $(\sigma^2_t)_{t\in [0,T]}$, 
 given $Z$. 
 Then, conditionally to $X$, the process $(Z_t)_{t\in [0,T]}$ has drift 
$$ 
 \int_0^t \frac{\sigma^2_s}{\tau^2_s+\sigma^2_s} dv_s 
 + 
 \int_0^t \frac{\tau^2_s}{\tau^2_s+\sigma^2_s} dX_s 
 \quad 
 \mbox{and variance} 
 \quad 
\int_0^t \frac{\tau^2_s\sigma^2_s}{\tau^2_s+\sigma^2_s} ds,
 \qquad t\in [0,T] 
. 
$$ 
\end{prop} 

\footnotesize 

\def\cprime{$'$} \def\polhk#1{\setbox0=\hbox{#1}{\ooalign{\hidewidth
  \lower1.5ex\hbox{`}\hidewidth\crcr\unhbox0}}}
  \def\polhk#1{\setbox0=\hbox{#1}{\ooalign{\hidewidth
  \lower1.5ex\hbox{`}\hidewidth\crcr\unhbox0}}} \def\cprime{$'$}

\end{document}